\documentclass[12pt,a4paper]{amsart} 
\calclayout
\usepackage{amssymb} 
\usepackage{mathtools}
\mathtoolsset{showonlyrefs}
\usepackage{mathrsfs} 
\usepackage{environ}
\usepackage{enumitem}
\usepackage{hyperref} 
\usepackage{amsrefs}
\theoremstyle{plain}
\newtheorem{theorem}{Theorem}[section]
\newtheorem{lemma}[theorem]{Lemma}
\newtheorem{proposition}[theorem]{Proposition}

\theoremstyle{remark}
\newtheorem{example}[theorem]{Example} 
\newtheorem{remark}[theorem]{Remark}
\newtheorem*{remark*}{Remark}

\DeclareMathOperator\supp{supp}
\DeclareMathOperator\diam{diam}
\DeclareMathOperator{\dist}{dist}
\numberwithin{equation}{section}
\allowdisplaybreaks
\newcounter{statement}
\newenvironment{statement}{\setcounter{statement}{\value{equation}} \list{(\theequation)}{\usecounter{equation}} \setcounter{equation}{\value{statement}} \item }{\endlist}
\newcounter{constant}
\newcommand{\C}{\refstepcounter{constant}\ensuremath{C_{\theconstant}}}
\newcommand{\Cr}[1]{\ensuremath{C_{\ref{#1}}}}
\newcommand{\R}{\mathbb{R}}
\newcommand{\N}{\mathbb{N}}
\newcommand{\ve}{\varepsilon}
\newcommand{\ex}{\mathbb E}

\newenvironment{note}{\color{gray}}{\ignorespacesafterend}
\RenewEnviron{note}{}

\begin{document}
\title[Coalescing Brownian motions]{On the coming down from infinity of coalescing Brownian motions}

\author[C. Barnes]{Clayton Barnes}
\address[C. Barnes]{The Faculty of Industrial Engineering and Management \\ Technion --- Israel Institute of Technology \\ Haifa 3200003\\ Israel}
\email{cbarnes@campus.technion.ac.il}

\author[L. Mytnik]{Leonid Mytnik}
\address[L. Mytnik]{The Faculty of Industrial Engineering and Management \\ Technion --- Israel Institute of Technology \\ Haifa 3200003\\ Israel}
\email{leonid@ie.technion.ac.il}

\author[Z. Sun]{Zhenyao Sun}
\address[Z. Sun]{The Faculty of Industrial Engineering and Management \\ Technion --- Israel Institute of Technology \\ Haifa 3200003\\ Israel}
\email{zhenyao.sun@gmail.com}

\begin{abstract}
	Consider a system of Brownian particles on the real line where each pair of particles coalesces at a certain rate according to their intersection local time. 
	Assume that there are infinitely many initial particles in the system. 
	We give a necessary and sufficient condition for the number of particles to come down from infinity. 
	We also identify the rate of this coming down from  infinity for different initial configurations. 
\end{abstract}
\maketitle

\section{Introduction} \label{:}
\subsection{Motivation and Background}
 Coming down from infinity is a property observed for certain stochastic population dynamics models, where the population begins with infinitely many members, and where members undergo competitive interactions that inevitably bring the entire population to a finite size in positive time. 
		Kingman's coalescent, perhaps one of the simplest such examples, is a continuous time Markov process taking 
	values in the set of partitions of $\mathbb N$. 
	If one thinks of blocks in a partition of $\mathbb N$ as particles, then Kingman's coalescent is a particle system where each pair of particles coalesce into one particle with unit rate, independently of other pairs. 
	If initially there are infinitely many particles, the coming down from infinity property says that almost surely, after any positive amount of time, there are only finitely many particles in the system. In 1999, Aldous showed that Kingman's coalescent comes down from infinity \cite{MR1673235}. 
	There are two goals when addressing the coming down from infinity for a given stochastic process. The first is to show whether, and under what conditions, the coming down from infinity property holds. The second, and perhaps more central, goal is then to find the rate with which the number of particles approaches infinity as time goes to zero. 
	
    Let $N_t$ be the number of particles at time $t> 0$ in Kingman's coalescent with infinitely many initial particles. 
	It is known  in \cite{MR1673235}  that $N_t$ behaves like $v(t):=2/t$ when $t\downarrow 0$.
	More precisely, it is shown in \cite{MR2599198} that,  $N_t/v(t)$ converges to $1$ as $t\downarrow 0$, almost surely, and in $L^p$ for any $p\geq 1$.
	Here the rate  function $v(t)$ arises as the solution to the nonlinear ordinary differential equation (ODE) with the singular initial condition:
\begin{equation}\label{:.0001}
\begin{cases}\displaystyle 
	\frac{\mathrm d}{\mathrm dt}v(t)
	=-\frac{1}{2}v(t)^2, 
	\quad t>0,
\\ \displaystyle 
	v(0)
	=\infty.
\end{cases}
\end{equation}
	For more references on coalescent theory, we refer our readers to \cite{MR2574323} and the references therein.

	The phenomenon of coming down from infinity is also considered in \cite{MR2892958} and \cite{MR2223040} for spatial Kingman's coalescent, extending the notion of Kingman's coalescent to the discrete spatial setting.
	In this model, each particle undergoes continuous-time simple random walk on   a connected graph with uniformly bounded degree,  and each pair of particles located at the same vertex coalesce into one particle with unit rate.
	It is proved in \cite{MR2892958} and \cite{MR2223040} that the spatial Kingman's coalescent comes down from infinity if and only if the underlying graph is finite.
	
	 In 1988, Shiga \cite{MR948717} proposed a model which is naturally considered an analogy of Kingman's coalescent in the continuum spatial setting. 
		In this particle system, particles move as independent Brownian motions on $\mathbb R$, and each pair of the particles coalesce into one particle with rate $1/2$ according to their intersection local time.
	  We  will refer to this model as \emph{(slowly) coalescing Brownian motions}. 
	 Coalescing Brownian motions  has drawn much attention due to its connection to the stochastic heat equation with Wright-Fisher white noise \cites{MR4235476, MR2014157, MR2793860, barnes2021effect, MR3968719, MR4278798, MR3582808, MR1813840, blath2022stochastic, MR3846839, MR1415234, MR2162813,  MR948717}.

	It is perhaps natural to ask the following question:
\begin{statement}
\label{:.001}
\emph{Does coalescing Brownian motions come down from infinity?  If it does, what is the rate of this coming down from infinity?}
\end{statement}
	Hobson and Tribe \cite{MR2162813} consider this question for coalescing Brownian motions on the unit circle $\mathbb S_1$, with initial particles sampled according to a Poisson point measure with intensity $n$ times the uniform measure of $\mathbb S_1$.
	Denote by $\hat N^{(n)}_t$ the total number of particles in the system at time $t\geq 0$.
	They proved that, as $n\uparrow \infty$, $\hat N^{(n)}_t$ has a finite weak limit $\hat N_t$ for every strictly positive $t$.
	They also showed that $\hat N_t/v(t)$ converges to $1$ in probability as $t\downarrow 0$ with the   rate  function $v(t)= 2/t$.
	However, their proofs rely on the compactness of the circle $\mathbb S_1$ that does not extend to  coalescing Brownian motions on the real line $\mathbb R$. 
	In this paper, we will give a  more satisfactory answer to \eqref{:.001}.

\subsection{Definition and Main Results}
	Let $\mathcal N$ be the space of locally finite atomic measures on $\mathbb R$ equipped with the vague topology.
	In this paper, the coalescing Brownian motions (with possibly infinite many initial particles) will be defined as $\mathcal N$-valued c\`adl\`ag Markov processes on $(0,\infty)$ whose entrance laws and transition probabilities will be specified later.
	To motivate that formal definition, however, we want to first construct an essentially equivalent particle system which describes the trajectory of each particles. 
	This construction is due to Tribe \cite{MR1339735}*{Section 2}, and the idea is that the particles are labeled by integers, and whenever a pair of particles coalesces, the particle with a larger label will be killed by the particle with a smaller label. 
 	
 	Let $I_0$ be the collection of labels of the initial particles: if there are finitely many initial particles, then $I_0=\{i: 1\leq i \leq n\}$ for some $n\in \mathbb N \cup \{0\}$; otherwise if there are infinitely many initial particles, then $I_0 = \mathbb N$.
		Denote by $x_i\in \mathbb R$ the location of the initial particle labeled by $i$. 
	If $I_0 = \varnothing$, then $(x_i)_{i\in I_0}$ is the empty list $\varnothing$, and nothing needs to be constructed because there is no initial particles.
	Otherwise if $I_0 \neq \varnothing$, the construction is formulated as follows.
\begin{statement}
\label{:.01}
	Let $\{(B^{(i)}_t)_{t\geq 0}:i\in I_0\}$ be a list of independent Brownian motions on $\mathbb R$ defined on a complete probability space $(\Omega, \mathscr F, \mathbb P)$ such that $B_0^{(i)} = x_i$ for each $i\in I_0$.
\end{statement}
\begin{statement}
	Let $\{\mathbf e^{(i)}: i\in I_0 \}$ be a family of i.i.d.\ exponential random variables with mean $2$, defined on $(\Omega, \mathscr F, \mathbb P)$, independent of the Brownian motions in \eqref{:.01}.
\end{statement}
\begin{statement}
	For each $i,j \in I_0$ with $i>j$, denote by $(L^{(i,j)}_t)_{t\geq 0}$ the local time at zero of the process $(B^{(i)}_t-B^{(j)}_t)_{t\geq 0}$.
\end{statement}
\begin{statement}
	Define $\zeta_1 := \infty$, and inductively for each $i\in I_0$ with $i>1$,
	\begin{equation}
		\zeta_{i} 
		:= \inf\left\{t\geq 0: \sum_{j=1}^{i-1} L^{(i,j)}_{t\wedge \zeta_j} \geq \mathbf e^{(i)}\right\}.
	\end{equation}
\end{statement}
	We call the stopping time $\zeta_i$ the lifetime of the $i$th particle, and we define 
\begin{equation}
X^{(i)}_t := \begin{cases}
	B^{(i)}_t, \quad &\text{if~} t\in [0,\zeta_i),
	\\ \dagger, \quad &\text{otherwise},
\end{cases}
\end{equation}
for every $i\in I_0$.
Here $\dagger$ is called the cemetery state and is not contained in $\mathbb R$.
	Now we say the set of random variables $\mathbf X:=\{X^{(i)}_t: t\geq 0, i \in I_0\}$ is a \emph{coalescing Brownian particle system} with initial configuration $(x_i)_{i\in I_0}$.
	The space of all the possible initial configurations is denoted by \[\mathcal X:= \{\varnothing\} \cup \bigcup_{n\in \mathbb N\cup\{\infty\}} \mathbb R^n.\]

	After we have constructed the trajectory of each particles, it is natural to consider the counting measure formed by the locations of the particles at a fixed time.
	Denote by $I_t:= \{i\in I_0: t\in [0,\zeta_i)\}$ the set of the labels of the particles alive at time $t\geq 0$. 
	For any open set $A\subset \mathbb R$, let $\mathscr B(A)$ be  Borel $\sigma$-algebra on $A$. 
	For every $U \in \mathscr B(\mathbb R)$ and $t\geq 0$, define a $\mathbb N\cup \{0,\infty\}$-valued random variable
\begin{equation} 
	 \mathtoolsset{showonlyrefs=false}
	\label{eq:ZtU}
		Z_t(U) := \sum_{i \in I_t} \mathbf{1}_U(X_t^{(i)})
\end{equation}
	which is the number of living particles contained in $U$ at time $t.$ 
	Let us also define a filtration $(\mathscr F_t)_{t\geq 0}$ so that, 
\begin{equation} \label{:filtration}
	\mathscr F_t := \sigma(Z_s(U): s\leq t, U \in \mathscr B(\mathbb R)),  \quad t\geq 0.
\end{equation}

	At this point, however, it is not clear whether $(Z_t)_{t>0}$ is a $\mathcal N$-valued process.
	Before we show that this is indeed the case, let us define 
\begin{equation}
\mathcal T_{\mathrm a} := \{(\Lambda, \mu): \text{$\Lambda$ is a closed subset of $\mathbb R$, $\mu$ is an atomic Radon measure on $\mathbb R\setminus \Lambda$}\},
\end{equation}
and  introduce a map $\Psi$ from $\mathcal X$ to $\mathcal T_{\mathrm a}$ so that $(\Lambda, \mu)= \Psi((x_i)_{i\in I_0})$ provided
	\begin{align} 
		\label{:lambda}
		&\Lambda 
		= \left\{y\in \mathbb R: \sum_{i\in I_0} \mathbf 1_{(y-r,y+r)}(x_i) = \infty, \forall r>0 \right\}, 
		\\ \label{:mu}
		&\mu(B) 
		= \sum_{i\in I_0} \mathbf 1_B(x_i), \quad \forall B \in \mathscr B(\Lambda^\mathrm c).
	\end{align}
	We call $\Psi((x_i)_{i\in I_0})$ the \emph{initial trace} of the coalescing Brownian particle system with initial configuration $(x_i)_{i\in I_0}$. 
	It is clear that $\mathcal T_{\mathrm a}$ is the collection of all the possible initial traces of the  coalescing Brownian particle systems, in the sense that $\Psi: \mathcal X \to \mathcal T_{\mathrm a}$ is a surjection.
	In what follows, we denote by $\supp(\Lambda,\mu) := \Lambda \cup \supp(\mu)$ the support of a given $(\Lambda,\mu) \in \mathcal T_{\mathrm a}$.
	Recall that a set $S \subset \R$ is said to be bounded if $\diam(S) := \inf\{K : |x - y| < K, \forall x,y \in S\}<\infty$.
We are now ready to state our first result.

\begin{theorem}\label{:I}
	Let $(\Lambda,\mu) = \Psi((x_i)_{i\in I_0})$ and let $U\subset \mathbb R$ be an arbitrary open interval.
\begin{enumerate}[label = \emph{(\roman*)}]
\item\label{:II} 
	If $U \cap \supp(\Lambda, \mu)$ is unbounded, then $\mathbb P( Z_t(U) = \infty, \forall t\geq 0 ) = 1$.
\item\label{:IF} 
		If $U \cap \supp(\Lambda,\mu)$ is bounded, then $\mathbb P(Z_t(U) < \infty, \forall t>0) = 1$. 
\end{enumerate}
\end{theorem}

	If $U$ is bounded, then Theorem \ref{:I}  \emph{\ref{:IF}} shows that $Z_t(U)$ is finite for all time $t>0$, almost surely.
	This implies that $(Z_t)_{t>0}$ is a $\mathcal N$-valued process. 
We will verify that $(Z_t)_{t>0}$ is a c\`adl\`ag Markov process and characterize its law through its duality with Wright-Fisher stochastic partial differential equation (Wright-Fisher SPDE) 
		\begin{equation} \label{:II.3}
			\partial_t u_{t,x} 
			= \frac{1}{2}\partial_x^2 u_{t,x} +   \sqrt{u_{t,x}(1-u_{t,x})} \dot W_{t,x}, 
			\quad t>0, x\in \mathbb R.
		\end{equation}	
	Let us briefly review some results about \eqref{:II.3} here.
	Denote by $\mathcal C_{[0,1]}$ the collection of $[0,1]$-valued continuous functions on $\mathbb R$, equipped with the topology of uniform convergence on compact sets.
	\begin{note}
		$\mathcal C_{\mathrm{tem}}$ is Polish \cite{MR2094150}.
	\end{note}
	Let $f$ be an arbitrary $[0,1]$-valued Borel measurable function on $\mathbb R$.
	According to \cite{MR1271224}, there exists a filtered probability space $(\mathbf \Omega, \mathscr G, (\mathscr G_t), \mathbf P_f)$, and on this space, an adapted  $\mathcal C_{[0,1]}$-valued continuous process   $(u_{t,\cdot})_{t>0}$  and a space-time white noise $W$, satisfying the mild form of the Wright-Fisher SPDE \eqref{:II.3} with the initial condition $u_0 = f$.
	Namely,  for every $(t,x)\in(0,\infty)\times \mathbb R$,
	\begin{equation} \label{:II.35}
		u_{t,x}
		= \int G_{t,x-y}f_{y}\mathrm dy + \iint_0^t G_{t-s,x-y} \sqrt{u_{s,y}(1-u_{s,y})} W(\mathrm ds\mathrm dy) \quad \text{a.s.} 
	\end{equation}
	Here, $G$ is the heat kernel given by $G_{t,x}:= e^{-x^2/(2t)}/\sqrt{2\pi t}$ for every $(t,x)\in (0,\infty)\times \mathbb R$, and the second term on the right hand side of \eqref{:II.35} is given by Walsh's stochastic integral  driven by a space-time white noise \cite{MR876085}. 
	\begin{note}
		\cite{MR863723} is also a good reference on Walsh's stochastic integral. 
	\end{note}
	The law of the $\mathcal C_{[0,1]}$-valued continuous process $(u_{t,\cdot})_{t> 0}$ is uniquely determined by its initial value $f$ \cite{MR948717}*{Theorem 5.1 (2)}. 

	Let us now give the formal definition of the coalescing Brownian motions.
	For any $(\Lambda,\mu) \in \mathcal T_{\mathrm a}$, we say that a process $(Y_t)_{t>0}$, living in a filtered probability space with filtration $(\mathscr F^Y_t)_{t>0}$ and probability $\mathbb P_{(\Lambda,\mu)}$, is a \emph{coalescing Brownian motions process}    with initial trace $(\Lambda,\mu)$, if the following statements hold.
\begin{statement}
	$(Y_t)_{t>0}$ is an $\mathcal N$-valued c\`adl\`ag Markov process.
\end{statement}
\begin{statement} \label{eq:EntranceLaw}
	For any $t>0$ and $[0,1]$-valued Borel function $f$ on $\mathbb R$,
	\begin{equation}
		\mathbb E_{(\Lambda,\mu)}[ \exp\left\{\langle \log (1-f), Y_t \rangle\right\}] = \mathbf E_f[\mathbf 1_{\{u_{t,x} = 1,\forall x\in \Lambda\}} \exp\left\{\langle \log (1-u_t),\mu\rangle \right\}].
	\end{equation}
\end{statement}
\begin{statement} \label{eq:TransitionProbability}
	For any $t>s>0$ and $[0,1]$-valued Borel function $f$ on $\mathbb R$,
	\begin{equation}
		\mathbb E_{(\Lambda,\mu)}\left[\exp\{\langle \log(1-f), Y_t\rangle \}\middle| \mathscr F^{Y}_s \right] 
		= \Theta^f_{t-s} (Y_s)
	\end{equation}
	where
	$
	\Theta^f_{t-s}(\kappa) : = 	\mathbf E_f[ \exp\{\langle \log(1-u_{t-s}), \kappa \rangle \}]$ for $\kappa \in \mathcal N$.
\end{statement}
Here, we denote by $\mathbb E_{(\Lambda,\mu)}$ and $\mathbf E_{f}$ the expectations for the probabilities $\mathbb P_{(A,\nu)}$ and $\mathbf P_f$ respectively. 
	Notice that \eqref{eq:EntranceLaw} and \eqref{eq:TransitionProbability} characterize the entrance laws and the transition probabilities of the coalescing Brownian motions through the Laplace transform.
	In particular, if such process $(Y_t)_{t> 0}$ exists, then its law is uniquely determined by its initial trace $(\Lambda,\mu)$.
	
	Our next result is to show the existence of such processes.
	
\begin{theorem} \label{thm:Markov}
	$~$
\begin{itemize}
\item[\emph{(i)}]
		Let $(x_i)_{i\in I_0} \in \mathcal X$, $(\Lambda,\mu) = \Psi((x_i)_{i\in I_0})$ and let
		 $\{Z_t(U):t\geq 0, U \in \mathscr B(\mathbb R)\}$ be random variables  in a filtered probability space $(\Omega, \mathscr F, \mathscr F_t, \mathbb P)$ constructed through \eqref{:.01}--\eqref{:filtration}. 
		Then the process $(Z_t)_{t>0}$  is a coalescing Brownian motions process with initial trace $(\Lambda, \mu)$.
\item[\emph{(ii)}] For any $(\Lambda,\mu)\in \mathcal T_a$, there exists a coalescing Brownian motions process with initial trace $(\Lambda,\mu)$.
\end{itemize}
\end{theorem}
\begin{remark}
\label{rem:2311_01}
Since $\Psi: \mathcal X \to \mathcal T_a$ is a surjection, for  any $(\Lambda,\mu)\in  \mathcal T_a $ there exist $I_0$ and $(x_i)_{i\in I_0} \in \mathcal X$ such that  $(\Lambda,\mu) = \Psi((x_i)_{i\in I_0})$. Thus part (ii) of Theorem~\ref{thm:Markov}  is an immediate corollary of part (i) of that theorem. 
\end{remark}

	To present our result on the coming down from infinity of the coalescing Brownian motions, we will discuss briefly the one-dimensional nonlinear partial differential equation
		 \[\partial_t v_{t,x} = \frac{1}{2}\partial_x^2v_{t,x}-\frac{1}{2}v_{t,x}^2, \quad t > 0, x\in \mathbb R.\]
	(See \cite{MR1429263}, and \cite{MR1697494} for a more comprehensive treatment.) 
	Denote by $\mathcal C^{1,2}((0,\infty) \times \mathbb R)$ the collection of functions $(h_{t,x})_{t>0,x\in \mathbb R}$ which is continuously differentiable in $t$ and twice continuously differentiable in $x$.
	For every open $U \subset \mathbb R$, denote by $\mathcal C_\mathrm c(U)$ the collection of continuous function whose support is a compact subset of $U$.
	According to \cite{MR1429263}*{Theorem 4}, for any closed set $A\subset \mathbb R$ and non-negative Radon measure $\nu$ on $A^\mathrm c$, there exists a unique non-negative $v^{(A, \nu)}\in \mathcal C^{1,2}((0,\infty)\times \mathbb R)$  such that 
\begin{equation}\label{:.05}
\begin{cases}\displaystyle
	\partial_t v^{(A,  \nu)}_{t,x} 
	= \frac{1}{2}\partial_x^2 v^{(A,  \nu)}_{t,x} - \frac{1}{2}\big(v^{(A,  \nu)}_{t,x}\big)^2, 
	\quad (t,x)\in (0,\infty)\times \mathbb R;
\\\displaystyle 
	\Big\{y\in \mathbb R: \forall r>0, \lim_{t\downarrow 0}\int_{y-r}^{y+r} v^{(A,  \nu)}_{t,x} \mathrm dx = \infty\Big\} = A;
\\\displaystyle
	\lim_{t\downarrow 0} \int \phi_x v^{(A,\nu)}_{t,x}\mathrm dx
	= \int \phi_x  \nu(\mathrm dx),
	\quad \phi \in \mathcal C_\mathrm c(A^\mathrm c).
\end{cases}
\end{equation}
	Notice that the PDE \eqref{:.05} is a spatial analogy of the ODE \eqref{:.0001}.	
	The pair $(A, \nu)$ is known as the  \emph{initial trace} of the solution $v^{(A, \nu)}$,  see also \cite{MR1697494} for more details. 

	The property of coming down from infinity of the coalescing Brownian motions with initial trace $(\Lambda,\mu)\in \mathcal T_\mathrm a$ is closely related to the solution  $v^{(\Lambda, \mu)}$ of PDE \eqref{:.05}.
		We observe this in our next result.
	\begin{theorem}\label{:E}
		Let $(\Lambda,\mu)\in \mathcal T_{\mathrm a}$ be arbitrary. 
		Suppose that $(Y_t)_{t> 0}$ is a coalescing Brownian motions process  with initial trace $(\Lambda,\mu)$ living in a filtered probability space with probability $\mathbb P_{(\Lambda,\mu)}$.
		Then the following two statements hold for arbitrary open interval $U\subset \mathbb R$. 
		\begin{enumerate}[label = \emph{(\roman*)}]
			\item \label{:StayInfinity}
			If $U \cap \supp(\Lambda, \mu)$ is unbounded, then $\mathbb P_{(\Lambda,\mu)}( Y_t(U) = \infty, \forall t> 0 ) = 1$.
			\item \label{:StayFinite}
				If $U \cap \supp(\Lambda,\mu)$ is bounded, then $\mathbb P_{(\Lambda,\mu)}(Y_t(U) < \infty, \forall t>0) = 1$. 
		\end{enumerate}
		Moreover, in the latter case when $U \cap \supp(\Lambda,\mu)$ is bounded, the following three statements hold. 
		\begin{enumerate}[label = \emph{(\roman*)}]
			\setcounter{enumi}{2}
			\item \label{:EK}
			$\mathbb E_{(\Lambda,\mu)}[Y_t(U)]<\infty$ for every $t>0$.
			\item\label{:EI} 
			If $\overline{U} \cap \Lambda= \varnothing$, then $\limsup_{t\downarrow 0}\mathbb E_{(\Lambda,\mu)}[Y_t(U)] < \infty$. 
			\item\label{:EF} 
			If $\overline{U} \cap \Lambda \neq \varnothing$, then as $t\downarrow 0$, $\mathbb E_{(\Lambda,\mu)}[Y_t(U)]\to \infty$ and
			\begin{align}
				\left(\int_U v^{(\Lambda, \mu)}_{t, x} \mathrm dx\right)^{-1} Y_t(U) 
				\longrightarrow 1 \text{~in~} L^1 \text{~w.r.t.~} \mathbb P_{(\Lambda,\mu)}. 
			\end{align}
		\end{enumerate}
	Here $\overline{U}$ is the closure of $U$.
	\end{theorem}
	By taking $U = \mathbb R$, the above theorem says that the total population in a coalescing Brownian motions is finite for all positive time, provided its initial trace $(\Lambda,\mu)$ is compactly supported.
	Conversely, if the initial trace is not compactly supported, then the total population remains infinite.
	This answers the first part of the question \eqref{:.001}.
	Furthermore, in the case when the initial trace is compactly supported and $\Lambda \neq \varnothing$, the total population $(Y_t(\mathbb R))_{t>0}$ comes down from infinity, and its small 
	times behavior is described ``approximately'' by $\int_{\mathbb R} v_{t, x}^{(\Lambda, \mu)}\mathrm d x.$
	This answers the second part of the question \eqref{:.001}. 
	Let us also mention that even if the initial trace is not compactly supported, Theorem~\ref{:E}(ii)  implies that  the population is \emph{locally} bounded at all positive times.

\subsection{Examples of rate functions}
		 Theorem \ref{:E} says that the total population comes down from infinity if and only if the initial trace $(\Lambda, \mu)$ is compactly supported and $\Lambda \neq \varnothing$. 
		In this case, the rate in which the total population comes down from infinity is given by the rate function $t\mapsto \|v_{t,\cdot}^{(\Lambda, \mu)}\|_{L^1(\mathbb R)}$. 
		It turns out that this rate function is related to the fractal structure of the set $\Lambda$. 
		In this section, we will give some examples where this rate can be explicitly calculated. 

		To present our result, we want to introduce the concept of Minkowski dimension. 
	 Let $\lambda$ denote Lebesgue measure on $\mathbb R$, and let $\# A$ denote the cardinality of a given set $A$. 
	Define the Minkowski sum 
	\[A +\tilde A := \{a+\tilde a: a\in A, \tilde a\in \tilde A\}\] 
	for any subsets $A$ and $\tilde A$ of $\mathbb R$.
	Define
	\[rA := \{ra:a\in A\}\]
	for any $r\in \mathbb R$ and $A\subset \mathbb R$.
\begin{note}
	Notice that if both $A$ and $\tilde A$ are compact then so is $A+\tilde A$. 
	Also notice that if one of $A$ or $\tilde A$ is open then so is $A+ \tilde A$.
\end{note}
	Denote by $B^\mathrm o:=(-1,1)$ the unit open ball centered at the origin. 
	For  every  $A\subset \mathbb R$, if there exists a (unique) $\delta \in [0,1]$ such that 
	\[1-(\log r)^{-1}\log\lambda(A+rB^\mathrm o) \to \delta\] 
	as $r\downarrow 0$, then we say $A$ has Minkowski dimension $\delta$.
	For many examples and applications on the Minkowski dimension, see \cite{MR2977849}.
\begin{note}
	Here we list some examples regarding the Minkowski dimension.
\begin{itemize}
\item
	If $A\subset \mathbb R$ has strictly positive finite cardinality $|A|$, then $A$ is Minkowski measurable with Minkowski dimension $0$ and Minkowski content $2|A|$.
\item
	If $A\subset \mathbb R$ is compact and has strictly positive Lebesgue measure, then $A$ is Minkowski measurable with Minkowski dimension $1$ and Minkowski content $\lambda(A)$.
\item
	Any self-similar set satisfying the open set condition has Minkowski dimension \cite{MR1694290}.
	In particular, the ternary Cantor set has Minkowski dimension $\log 2/\log 3$.
	However, the ternary Cantor set is not Minkowski measurable \cite[Section 4.1]{MR1694290}.
\item
	Any non-lattice self-similar set satisfying the open set condition is Minkowski measurable \cite[Theorem 2.3 (1)]{MR1694290}. 
\end{itemize}
	See the book \cite{MR2977849} for many more examples and discussions.
\end{note}

		The proof of the following Proposition is postponed in Section \ref{:B:}.

\begin{proposition} \label{:B}
	Let $A$ be a compact subset of $\mathbb R$.
\begin{enumerate}
\item 
	If $A$ has positive finite cardinality, then $\sqrt{t}\|v^{(A, \mathbf 0)}_{t,\cdot}\|_{L^1(\mathbb R)}$  converges to $\Cr{:.1}\# A$ as $t\downarrow 0$ where $\C\label{:.1}:=\|v^{(\{0\}, \mathbf 0)}_{1,\cdot}\|_{L^1(\mathbb R)}\in (0,\infty)$.
\item
	If $A$ has positive Lebesgue measure, then $t\|v^{(A,\mathbf 0)}_{t,\cdot}\|_{L^1(\mathbb R)}$ converges to $2\lambda(A)$ as $t\downarrow 0$.
\item 
	If  $A$ has Minkowski dimension $\delta \in (0,1)$, then $(\log t)^{-1}\log\|v^{(A,\mathbf 0)}_{t,\cdot}\|_{L^1(\mathbb R)}$ converges to $-(1+\delta)/2$ as $t\downarrow 0$. 
\end{enumerate}
\end{proposition}

	\begin{example}[Infinitely many initial particles at one point]\label{:C}
		For instance, assume that there are infinitely many initial particles which are all located at the origin, i.e. $x_i = 0$ for every $i\in \mathbb N$. 
		In this case $\Lambda$ is the singleton $\{0\}$ and $\mu$ is the null measure $\mathbf 0$ on $\mathbb R\backslash\{0\}$.
		Now Theorem \ref{:E}  \emph{\ref{:EF}} and Proposition \ref{:B} says that that $\sqrt{t}Y_t(\mathbb R)$ converges to $\Cr{:.1}$ in $L^1$ as $t\downarrow 0$. 
	\end{example}

\begin{remark*}
	In Example \ref{:C},  the total population $Y_t(\mathbb R)$ is comparable to $\Cr{:.1}/\sqrt{t}$, which decreases even faster than the   rate  function   $2/t$  of Kingman's coalescent.
	One might find this surprising since one might expect the spatial movement of the particles, if not slowing down the coalescing, should not speed it up.
	However, this is only an illusion due to the different clocks used by the two models for their coalescing mechanisms---one uses the ordinary clock while the other uses the local time.
	For a better comparison, one might choose to observe $(Y_t(\mathbb R))_{t\geq 0}$ according to the clock of the local time.
	Recall that, if $(L_t)_{t\geq 0}$ is the local time of a Brownian motion at zero, then $\mathbb E[L_t]=\sqrt{2t/\pi}$.
	Denote by $\tilde N_l$ the total population $Z_t(\mathbb R)$ when the expected local time $\mathbb E[L_t]$ is at the level $l$.
	Now   from Example \ref{:C} we have  that $\tilde N_l$ is of order $1/l$ for small $l$, which behaves similar to Kingman's coalescent.
\end{remark*}

\begin{example}[$\Lambda$ with positive Lebesgue measure]\label{:P}
	Suppose that $\Lambda$ is compact and has positive finite Lebesgue measure and $\mu$ is the null measure supported on $\Lambda^c$. 
	Then by Proposition \ref{:B} (2) we have 
\(
	t  Y_t (\mathbb R) \overset{L^1}{\to} 2\lambda(\Lambda)
\)
	as $t\downarrow 0$. 
\end{example}
\begin{note}
	Here we list some examples regarding the Minkowski dimension.
\begin{itemize}
\item
	If $A\subset \mathbb R$ has strictly positive finite cardinality $|A|$, then $A$ is Minkowski measurable with Minkowski dimension $0$ and Minkowski content $2|A|$.
\item
	If $A\subset \mathbb R$ is compact and has strictly positive Lebesgue measure, then $A$ is Minkowski measurable with Minkowski dimension $1$ and Minkowski content $\lambda(A)$.
\item
	Any self-similar set satisfying the open set condition has Minkowski dimension \cite{MR1694290}.
	In particular, the ternary Cantor set has Minkowski dimension $\log 2/\log 3$.
	However, the ternary Cantor set is not Minkowski measurable \cite[Section 4.1]{MR1694290}.
\item
	Any non-lattice self-similar set satisfying the open set condition is Minkowski measurable \cite[Theorem 2.3 (1)]{MR1694290}. 
\end{itemize}
	See the book \cite{MR2977849} for many more examples and discussions.
\end{note}
\begin{example}[$\Lambda$ with known Minkowski dimension] \label{:T}
	Suppose that  $\Lambda$  is compact and has Minkowski dimension $\delta\in (0,1)$.
	Suppose also that $\mu$ is the null measure supported on $\Lambda^c$. 
	Then by Proposition \ref{:B} (3) we have $(\log t)^{-1}\log  Y_t(\mathbb R)$ converges to $-(1+\delta)/2$ in probability as $t\downarrow 0$.
\end{example}

\begin{note}
	\begin{issue}
		CB: Were we keeping or removing this conjecture? Can't recall what we discussed.
		ZS: Let us remove it.
	\end{issue}
	Observing from   Examples  \ref{:C}-\ref{:T}, it is perhaps reasonable to make the following conjecture. 
\begin{conjecture}
 	Suppose that $\Lambda$ is a Minkowski measurable set with dimension $\delta \in [0,1]$ and content $\kappa\in (0,\infty)$.
 	Then there exists a deterministic constant $\C\label{:.3}(\delta)>0$, depending only on $\delta$, such that   $t^{(1+\delta)/2} Z_t(\mathbb R)$  converges to $\Cr{:.3}(\delta)\kappa$ in $L^1$ as $t\downarrow 0$. 
\end{conjecture}
\end{note}

\subsection{Proof Strategy}
	Instead of comparing coalescing Brownian motions with the PDE \eqref{:.05} directly, the idea is to consider their dual counterpart. There are two duality relations used in this paper.
	One is the moment duality given by Shiga \cite{MR948717} between coalescing Brownian motions and 
	the Wright-Fisher SPDE \eqref{:II.3}.
	The other is the classical duality established between the PDE \eqref{:.05} and the super-Brownian motion as well as the Brownian snake (see \cite{MR1429263}).
	Brownian snake was introduced by Le Gall \cite{MR1207305} to study the genealogical structure of the 
	Super-Brownian motion, whose density, in the one-dimensional case satisfies the following SPDE (see \cite{MR958288}, \cite{MR983088})
\[
	\partial_t \tilde u_{t,x} = \frac{1}{2}\partial_x^2 \tilde u_{t,x} + \sqrt{\tilde u_{t,x}} \dot W_{t,x}, \quad t> 0, x\in \mathbb R.
\]
Observe that the SPDE above is quite similar to \eqref{:II.35}, in the sense that their noise coefficient are very close to each other for small values of $u$. 
	In particular, one expects that with high probability the non-negative random fields $(u_{t,x})$ and $(\tilde u_{t,x})$ stay close to each other on finite time intervals if they share the same small initial conditions $u_0 = \tilde u_0 \ll 1$.  
	
	In the actual proof, we build a connection between Wright-Fisher SPDE \eqref{:II.3} and the PDE \eqref{:.05} by searching for the Doob-Meyer decomposition of the process
\[
	s\mapsto \exp \left(- \int u_{s,y}v_{t-s,y}^{(\Lambda, \mu)}\mathrm dy\right), \quad s\in [0,t).
\]
	This idea was already explored by Tribe \cite{MR1339735} where he established the compact support property of the Wright-Fisher SPDE \eqref{:II.3}.  
	The challenge here is that it is not clear how to apply It\^o's formula directly to the above process up to time $t$, because $v_{t-s,y}$ approaches a singular value when $s\to t$ and $y\in \Lambda$.
	So we will use a sequence of so-called non-singular initial traces $\{(\varnothing, \mu^{n}):n\in \mathbb N\}$ to approximate the singular initial trace $(\Lambda, \mu)$ (trace with $\Lambda\not=\emptyset$), following the techniques developed in  \cite{MR1697494}. 
\subsection{Paper Outline}
	The rest of the paper is organized as follows. 
	In Section \ref{:I:}, we give the proof of Theorems \ref{:I} and \ref{thm:Markov}. 
	In Section \ref{:E:}, we give the proof of Theorem \ref{:E}.
	In Section \ref{:B:}, we give the proof of Proposition \ref{:B}.

\subsection*{Acknowledgements} 
	The work of the authors  was supported in part by ISF grants No.~1704/18 and 1985/22. 
	The first author is a Zuckerman Postdoctoral Scholar, and this work was supported in part by the Zuckerman STEM Leadership Program.
	The third author is supported in part at the Technion by a fellowship of the Israel Council for Higher Education. 
	We want to thank Omer Angel, Julien Berestycki, Pascal Maillard, Michel Pain, and Eviatar Procaccia for helpful conversations.

\section{Proof of Theorems \ref{:I} and \ref{thm:Markov}} \label{:I:}
\subsection{Initial traces}
	Let us first review some concepts and notations from \cite{MR1697494}. 
	Denote by $\mathcal T$ the collection of pairs $(A,\nu)$ where $A$ is a closed subset of $\mathbb R$ and $\nu$ is a non-negative Radon measure on $A^\mathrm c$. 
	Then $\mathcal T$ is the space of all the possible initial traces for the PDE \eqref{:.05}.
	For any $(A,\nu)$ and $(\tilde A,\tilde \nu)$ in $\mathcal T$, we say $(A,\nu)\preceq (\tilde A,\tilde \nu)$ if  $A\subset \tilde A$ and $\nu(B) \leq \tilde\nu(B)$ for every $B\in \mathscr B(\tilde A^\mathrm c)$.
	Note that $\preceq$ is a partial order on $\mathcal T$.
	Denote by $\mathscr M_{\text{reg}}^+$ the space of outer regular (not necessary locally bounded) Borel measures on $\mathbb R$. 
	Define a map $\eta:(A,\nu)\mapsto \eta^{(A,\nu)}$ from $\mathcal T$ to $\mathscr M_{\text{reg}}^+$ so that for every $B \in \mathscr B(\mathbb R)$,
\begin{equation} \label{:Ieta}
	\eta^{(A,\nu)}(B) 
	=
\begin{cases}
	\infty,
	& \quad B\cap A \neq \varnothing,
	\\ \nu(B), 
	&\quad B \cap A = \varnothing.
\end{cases}
\end{equation}
\begin{note}
	Notice that the following statement hold:
\begin{itemize}
\item
	$\eta^{(A,\nu)}(B)\geq 0$ for any $B\in \mathscr B(\mathbb R)$.
\item
	$\eta^{(A,\nu)}(\emptyset) = 0$.
\item
	For disjoint sequence $(B_k)_{k\in \mathbb N}$ in $\mathscr B(\mathbb R)$,
\begin{itemize}
\item
	if there exists $k\in \mathbb N$ such that $B_k\cap A \neq \emptyset$, then
\[
	\eta^{(A,\nu)}(\cup_{k=1}^\infty B_k)
	= \infty 
	= \sum_{k=1}^\infty \eta^{(A,\nu)}(B_k), 
\]
\item
	if $B_k \cap A = \emptyset$ for all $k\in \mathbb N$, then
\[
	\eta^{(A,\nu)}(\cup_{k=1}^\infty B_k) 
	= \nu(\cup_{k=1}^\infty B_k)
	=\sum_{k=1}^\infty \nu(B_k) 
	= \sum_{k=1}^\infty \eta^{(A,\nu)}(B_k).
\]
\end{itemize}
\item
	For any $B\in \mathscr B(\mathbb R)$,
\begin{itemize}
\item
	if $B\cap A \neq \varnothing$, then 
\[
	\eta^{(A,\nu)}(B) 
	= \infty 
	= \inf \{\eta^{(A,\nu)}(U): \text{~$U$ is an open set containing $B$}\};
\]
\item
	if $B\cap A = \varnothing$, then
\begin{align}
	&\eta^{(A,\nu)}(B) 
	= \nu(B) 
	= \inf \{\nu(U): \text{~$U$ is an open subset of $A^\mathrm c$ containing $B$}\}  
	\\&= \inf \{\eta^{(A,\nu)}(U): \text{~$U$ is an open subset of $A^\mathrm c$ containing $B$}\}
	\\&\geq \inf \{\eta^{(A,\nu)}(U): \text{~$U$ is an open set containing $B$}\}
	\geq \eta^{(A,\nu)}(B).
\end{align}  
\end{itemize}
\end{itemize}
	Therefore, we know that $\eta^{(A,\nu)}$ must be an element of $\mathscr M_{\text{reg}}^+$.
\end{note}
	For any $\eta \in \mathscr M_{\text{reg}}^+$, denote by $\mathcal S_\eta$ the set of singular points of $\eta$, i.e.
\[
	\mathcal S_\eta 
	:= \{x\in \mathbb R: \eta(U) = \infty \text{~for every open neighborhood $U$ of $x$}\},
\]
	and by $\mathcal R_\eta := \mathcal S_\eta^c$ the set of regular points of $\eta$.
	For an $\mathscr M_{\text{reg}}^+$-sequence $(\eta_n)_{n\in \mathbb N}$, we say it converges m-weakly to an $\eta \in \mathscr M_{\text{reg}}^+$ if the following two conditions hold:
\begin{statement}
	If $U$ is an open subset of $\mathbb R$ with $\eta(U)=\infty$ then $\lim_{n\to\infty}\eta_n(U) = \infty$.
\end{statement}
\begin{statement}
	For every compact $K \subset \mathcal R_\eta$, the sequence $(\eta_n(K))_{n\in \mathbb N}$ is eventually bounded and $\lim_{n\to \infty} \int \phi \mathrm d\eta_n = \int \phi \mathrm d\eta$ for every $\phi \in \mathcal C_\mathrm c(\mathcal R_\eta)$.
\end{statement}

	Let us now list some properties of the solutions to \eqref{:.05} given their initial traces. 
	The following two statements can be verified from \cite{MR1697494}*{Proposition 3.10} and \cite{MR1429263}*{Theorem 4}.
\begin{statement}\label{:II.522}
	For every $(A,\nu),(\tilde A,\tilde \nu) \in \mathcal T$, if $(A,\nu)\preceq (\tilde A,\tilde \nu)$, then $v^{(A,\nu)}_{t,x} \leq v^{(\tilde A,\tilde \nu)}_{t,x}$ for every $t>0$ and $x\in \mathbb R$. 
\end{statement}
\begin{statement} \label{:II.521}
	For any $\mathcal T$-sequence $((A_n,\nu_n))_{n\in \mathbb N}$ and $(A,\nu)\in \mathcal T$, if $\eta^{(A_n,\nu_n)}$ converges m-weakly to $\eta^{(A,\nu)}$, then $v^{(A_n,\nu_n)}_{t,x}$ converges to $v^{(A,\nu)}_{t,x}$ for every $t>0$ and $x\in \mathbb R$.
\end{statement}
	The following three statements can be verified  directly from  the uniqueness of the solution to \eqref{:.05}.
\begin{statement}\label{:II.523}
	$v^{(\mathbb R,\mathbf 0)}_{t,x}= 2/t$ for every $t>0$ and $x\in \mathbb R$.
\end{statement}
\begin{statement}\label{:II.524}
	$ v_{t,x}^{(A,\mathbf 0)} = v^{(A+\{z\},\mathbf 0)}_{t,x+z}$ for every closed $A\subset \mathbb R$, $t>0$, $x\in \mathbb R$, and $z\in\mathbb R$.
\end{statement}
\begin{statement} \label{:II.525}
	$ v_{t,x}^{(A,\mathbf 0)} = v_{t,-x}^{(-A,\mathbf 0)} $  for every  closed $A\subset \mathbb R$, $t>0$ and $x\in \mathbb R$. 
\end{statement}
	According to \cite{MR4029158}*{Proposition 3.2 \& Lemma 3.4}, there exists a constant $\C\label{:II.53}\geq 2$ such that  
\begin{equation}\label{:II.526}
	v^{((-\infty,0],\mathbf 0)}_{t,x} 
	\leq \frac{\Cr{:II.53}}{t}\left(1+\frac{x}{\sqrt{t}}\right)e^{-\frac{x^2}{2t}},
	\quad t>0,x\geq 0.
\end{equation}
	Using \eqref{:II.522}, \eqref{:II.523}, \eqref{:II.524}, \eqref{:II.525} and \eqref{:II.526}, it is straightforwad to verify that  
\begin{align}\label{:II.54}
	v_{t,x}^{([-k,k],\mathbf 0)}
       \leq \frac{\Cr{:II.53}}{t}\left(1+\frac{d(x,[-k,k])}{\sqrt{t}}\right)e^{-\frac{d(x,[-k,k])^2}{2t}},
       \quad k\geq 0, t > 0, x\in \mathbb R,
\end{align}
	where $d(x,[-k,k]):= \inf\{|x-z|:z\in [-k,k]\}$.
\begin{note}
	In fact,
\begin{align}
	v_{t,x}^{([-k,k],\mathbf 0)}
	\leq v^{([-2k,0],\mathbf 0)}_{t,x-k}
	\leq v^{((-\infty,0],\mathbf 0)}_{t,x-k}
	\leq \frac{\Cr{:II.53}}{t}\Big(1+\frac{d(x,[-k,k])}{\sqrt{t}}\Big)e^{-\frac{d(x,[-k,k])^2}{2t}}.
\end{align}
\end{note}

\subsection{Proof of Theorem \ref{:I} \emph{\ref{:II}}} \label{:II:}
	For two lists of real numbers  $(x_i)_{i\in I_0}\in \mathcal X$ and $(\tilde x_i)_{i\in \tilde I_0} \in \mathcal X$, we say $(\tilde x_i)_{i\in \tilde I_0} $ is a sublist of $(x_i)_{i\in I_0}$ if there is an strictly increasing map $\iota$ from $\tilde I_0$ to $I_0$ so that  $\tilde x_i = x_{\iota(i)}$ for every $i\in \tilde I_0$.
\begin{lemma} \label{:IIS}
	Let $(\tilde x_i)_{i\in \tilde I_0}$ be a sublist of $(x_i)_{i \in I_0}$.
	Suppose that $\tilde {\mathbf{X}}=\{\tilde X^{(i)}_t: t\geq 0, i\in \tilde I_0\}$ is a coalescing Brownian particle system with initial configuration $(\tilde x_i)_{i\in \tilde I_0}$.
	For every $U\in \mathscr B(\mathbb R)$ and $t\geq 0$, denote by $\tilde Z_t(U)$ the number of living particles contained in $U$ at time $t$ in the particle system $\tilde {\mathbf{X}}$.
	Then for every $U\in \mathscr B(\mathbb R)$, the process $(\tilde Z_t(U))_{t\geq 0}$ is stochastically dominated by the process $(Z_t(U))_{t\geq 0}$. 
\end{lemma}
	Lemma \ref{:IIS} can be verified by a straightforward coupling method. We omit the details.
\begin{proof}[Proof of Theorem \ref{:I}\ref{:II}]
	We construct another particle system that is easily seen to have infinitely many particles  in $U$  for all time, and  is dominated by that of     coalescing Brownian motions.
	Notice that $ \supp(\Lambda, \mu)$ is also the support of the set $\{x_i:i\in I_0\}$.
	The condition of the theorem implies that $ \supp(\Lambda, \mu)$ is unbounded, and therefore there are infinitely many initial particles, i.e. $I_0 = \mathbb N$.
	Since   $U$ is an interval and $U\cap \supp(\Lambda,\mu)$ is   unbounded,  there exists a subsequence $(\tilde x_i)_{i\in \mathbb N}$ of $(x_i)_{i\in \mathbb N}$ such that  $\{ (\tilde x_i-2, \tilde x_i+2): i\in \mathbb N\}$ is a family of disjoint subsets of $U$. 
	Let $\tilde{\mathbf X}$ be a  coalescing Brownian particle system with initial configuration $(\tilde x_i)_{i\in \mathbb N}$. 
	Let $\tilde Z_t(U)$ be the number of the particles contained in $U$ at time $t\geq 0$ in the particle system $\widetilde{\mathbf X}$.
	From Lemma \ref{:IIS}, $(\tilde Z_t(U))_{t\geq 0}$ is stochastically dominated by  $(Z_t(U))_{t\geq 0}$.
	Now we only have to show that $\tilde Z_t(U) = \infty$ for every $t\geq 0$ almost surely.
	
	To do that, we want to construct another Brownian particle system.
	Let $\{(\hat B_t^{(i)})_{t\geq 0}:i\in \mathbb N\}$ be a sequence of independent Brownian motions such that $\hat B_0^{(i)} = \tilde x_i$ for each $i\in \mathbb N$. 
	Define $\hat \zeta_i:= \inf\{t\geq 0: |\hat B^{(i)}_t-\hat B^{(i)}_0|\geq 1\}$ and 
	\begin{equation}
		\hat X_t^{(i)} = \begin{cases}
			\hat B^{(i)}_t, \quad &\text{if~} t\in [0,\hat \zeta_i),
			\\ \dagger, \quad &\text{otherwise},
		\end{cases}
	\end{equation}
		for each $i\in \mathbb N$.
	Now $ \hat{\mathbf X}:=\{\hat X_t^{(i)}:t\geq 0, i\in \mathbb N\}$ is a system of independent Brownian particles where each particle is killed when the distance it travels reaches  $1$, that is, the $i$-th particle is killed at time $\hat \zeta_i$. 
	Denote by $\hat Z_t(U)$ the number of the particles contained in $U$ at time $t\geq 0$ in the particle system $\hat{\mathbf X}$.
	Since the initial particles in $\hat{\mathbf X}$ are located away from each other with distance at least $4$, each particle in $\hat{\mathbf X}$ is killed before it can meet with any other particles.
	Therefore  $(\hat Z_t(U))_{t\geq 0}$ is stochastically dominated by $(\tilde Z_t(U))_{t\geq 0}$.
	
	To finish the proof, we only have to show that $\hat Z_t(U) = \infty$ for every $t\geq 0$ almost surely.
		First observe that $\hat Z_t(U)$ is the total population of the system $\hat{\mathbf X}$, since for every $i\in\mathbb N$, particle $i$ are killed before it can leave $[x_i-1,x_i+1]\subset U$.
		In particular, we have $\hat Z_t(U) = \sum_{i\in \mathbb N} \mathbf 1_{\{\hat \zeta_i > t\}}$ which is a non-increasing process in $t\geq 0$.
	Also observe that $(\hat \zeta_i)_{i\in \mathbb N}$ is a family of i.i.d. random variables with $\mathbb P(\hat \zeta_i > t)>0$ for every $t \geq 0$.
	Now the desired result follows from the second Borel-Cantelli lemma.
\end{proof}

\subsection{Proof of Theorem \ref{:I}\emph{\ref{:IF}}} \label{:IF:}

	Recall that  $\mathbf X=\{X_t^{(i)}: t\geq 0, i\in I_0\}$ is a coalescing Brownian particle system with initial configuration $(x_i)_{i\in I_0}$ and initial trace $(\Lambda, \mu) = \Psi((x_i)_{i\in I_0})$. 
	If there are finitely many initial particles, i.e. $\# I_0 < \infty$, then the statement in Theorem \ref{:I}\emph{\ref{:IF}} holds automatically because $Z_t(U) \leq \# I_0$ by definition.
	Therefore, in the rest of this subsection, we assume that there are infinitely many initial particles, i.e. $I_0 = \mathbb N$.

	Let $\theta$ be a non-negative increasing continuous function from $[0, 1)$ to $[1,\infty)$ such that $\theta(0) = 1$ and $\theta(\gamma) = -\log(1-\gamma)/\gamma$ for every $\gamma \in (0, 1)$. 
Let us state several lemmas that will be proved later in this subsection.

\begin{lemma}\label{:IFP_1}
	Let $U\subset \mathbb R$ be an open interval and $\varepsilon \in (0,1)$.
	Let $(u_{t,x})_{t\geq 0, x\in \mathbb R}$ be the solution to the Wright-Fisher SPDE \eqref{:II.3} with initial condition $u_0 = \varepsilon \mathbf 1_U$.
	Let $F$ be a closed interval containing the set $\{x_i:i\in \mathbb N\}$. 
	Then for any $t>0$ and $\gamma \in (\varepsilon, 1)$, we have
\begin{equation}\label{:IFP_11}
	\mathbb E\left[(1 - \ve)^{Z_t(U)}\right] 
	\geq \mathbf E_{\varepsilon \mathbf 1_U}\left[\exp\left(-\theta(\gamma)\sum_{i = 1}^{\infty}u_{t,x_i}\right)\right] - \mathbf E_{\varepsilon \mathbf 1_U}\left(\sup_{s \leq t, y \in F} u_{s, y} > \gamma\right),
\end{equation}
	and
\begin{equation} \label{:IFP_12}
	\mathbb E\big[(1 - \ve)^{Z_t(U)}\big]
	\leq \mathbf E_{\ve \mathbf 1_U}\left[\exp\left(-\sum_{i = 1}^{\infty}u_{t,x_i}\right)\right].
\end{equation}
\end{lemma}

	For every $t> 0$, closed interval $F$, and $(A,\nu)\in \mathcal T$, define 
\begin{equation} \label{:IFP_15}
	\mathcal V^{(A,\nu,F)}_t 
	:= \int_0^t   \int_{F^{\mathrm c}}  \left(v^{(A,\nu)}_{r,z}\right)^2 \mathrm dz\mathrm dr.
\end{equation}

\begin{lemma}\label{:IFP_2}
	Let  $t > 0$.
\begin{enumerate}
\item
	If $U$ is an open interval so that $U \cap \{x_i:i\in \mathbb N\}$ is bounded, then $\int_U v^{(\Lambda, \mu)}_{t, y}\mathrm dy$ is finite.
\item
	If $F$ is a closed interval containing $\cup_{i\in \mathbb N}(x_i-1,x_i+1)$, then $\mathcal V^{(\Lambda,\mu,F)}_t$ is finite.
\end{enumerate}
\end{lemma}
\begin{lemma}\label{:IFP_25}
	Let $U$ be an open interval. 
	Let $F$ be a closed interval containing $\cup_{i\in \mathbb N}(x_i-1,x_i+1)$. 
	Let $t>0$   and $0\leq \ve \leq \gamma < 1$. 
	Suppose that $(u_{t,x})_{t\geq 0, x\in \mathbb R}$ is a solution to the Wright-Fisher SPDE with initial condition $u_0 = \ve \mathbf 1_U$.
	Then
\begin{align}
	& \mathbf E_{\varepsilon \mathbf 1_U}\left[\exp\left(-\theta(\gamma)\sum_{i = 1}^\infty u_{t, x_i}\right)\right]
		\\\label{:IFP_251}&\geq \exp\left(-\frac{\ve}{1- \gamma}\int_U v^{(\Lambda, \mu)}_{t,y} \mathrm d y\right) - \frac{\ve}{2(1 - \gamma)}\mathcal V^{(\Lambda, \mu,F)}_t -  \mathbf P_{\varepsilon \mathbf 1_U}\left(\sup_{s \leq t, y \in F} u_{s, y} > \gamma\right),
\end{align}
	and 
\begin{equation} \label{:IFP_252}
	\mathbf E_{\varepsilon \mathbf 1_U}   \left[\exp\left(-\sum_{i=1}^\infty u_{t,x_i}\right)\right]
	\leq \exp\left(-\varepsilon \int_U v^{(\Lambda,\mu)}_{t,y} \mathrm dy\right).
\end{equation}
\end{lemma}
\begin{lemma}\label{:IFP_3}

	Let $U$ be an open interval, and $F$ be a closed interval. 
	Let $t> 0$ and $0< \gamma < 1$.
	Suppose that $U \cap F$ is bounded.
	Then
\[
	\C\label{:IFP.31}(U, F, t, \gamma)
	:= \sup_{\ve \in (0,\gamma/2)} \frac{1}{\ve}   \mathbf P_{\varepsilon \mathbf 1_U} \left(\sup_{s \leq t, y \in F} u_{s, y} > \gamma\right) 
	< \infty.
\]
\end{lemma}
	We use the above lemmas to prove the next proposition.
	\begin{proposition}\label{:IF_fixed_t}
	Suppose that $F$ is a closed interval containing $\cup_{i\in \mathbb N} (x_i-1,x_i+1)$.
	Suppose that $U$ is an open interval such that $U\cap F$ is bounded.
	Then for any $t>0$ and $\gamma \in (0,1)$,
\begin{equation}\label{:IF_fixed_t1}
\mathbb E[Z_t(U)] \leq  \frac{1}{1 - \gamma}\left[\int_Uv^{(\Lambda, \mu)}_{t, y} \mathrm d y + \frac{1}{2}\mathcal V^{(\Lambda,\mu,F)}_t \right] + 2\Cr{:IFP.31}(U, F, t, \gamma)<\infty.\end{equation}
\end{proposition}
\begin{proof}
	Let $\varepsilon \in (0, \gamma/2)$ be arbitrary and $(u_{t,x})_{t\geq 0,x\in \mathbb R}$ be a solution to the Wright-Fisher SPDE \eqref{:II.3} with initial condition $u_0 = \varepsilon \mathbf 1_U$.
	From the condition of the proposition, it is easy to see that $U \cap \{x_i:i\in \mathbb N\}$ is bounded.
	By Lemmas \ref{:IFP_1}, \ref{:IFP_25} and \ref{:IFP_3}, we have
\begin{equation}
	\mathbb E\left[(1- \ve)^{Z_t(U)}\right] 
	\geq \exp\left(-\frac{\ve}{1- \gamma}\int_U v^{(\Lambda, \mu)}_{t,y} \mathrm d y\right) - \frac{\ve}{2(1 - \gamma)}\mathcal V^{(\Lambda, \mu,F)}_t - 2\ve \Cr{:IFP.31}(U, F, t, \gamma).
\end{equation}
	We extract the first moment of $Z_t(U)$ by taking the derivative of its moment generating function, that is,
\begin{align}
	&\mathbb E[Z_t(U)] 
	= \lim_{\ve \to 0^+}\frac{1 - \ex\left((1- \ve)^{Z_t(U)}\right)}{\ve}
\\
	&\leq \frac{1}{1 - \gamma}\left[\int_Uv^{(\Lambda, \mu)}_{t, y} \mathrm d y + \frac{1}{2}\mathcal V^{(\Lambda,\mu,F)}_t \right] + 2\Cr{:IFP.31}(U, F, t, \gamma),
\end{align}
	which is finite for every $t > 0$ by Lemmas \ref{:IFP_2} and \ref{:IFP_3}.
\end{proof}

	\begin{remark} \label{thm:Finitem}
	Let $U$ be an open interval such that $U \cap \supp(\Lambda,\mu)$ is bounded.
	Take $F$ to be the smallest closed interval containing $\cup_{i\in \mathbb N}(x_i-1,x_i+1)$.
	Since $\supp(\Lambda,\mu)=\mathrm{cl}(\{x_i:i\in \mathbb N\})$, where $\mathrm{cl}(A)$ denotes the closure of a set $A$,
	it is clear that $U\cap F$ is bounded.
Now the above proposition implies that for fixed $t > 0$, $\mathbb E[Z_t(U)] < \infty$ and therefore $\mathbb P(Z_t(U)< \infty) = 1$. 
\end{remark}

\begin{note} 
\begin{lemma} \label{:BS}
	Let $\{(\beta^{(i)}_t)_{t\geq 0}:i\in \mathbb N\}$ be a sequence of independent standard Brownian motions.
	Let $(X_i)_{i\in \mathbb N}$ be a sequence of random variables independent of the Brownian motions $\{\beta^{(i)}:i\in \mathbb N\}$, such that
\begin{equation}\label{:BS.1}
	\sup_{a\in \mathbb R}\mathbb E\left[\sum_{i\in \mathbb N}\mathbf 1_{\{X_i\in [a,a+1)\}}\right] < \infty.
\end{equation}
	Then, for every $T\geq 0$, 
\[
	\sup_{a\in \mathbb R}\mathbb E\left[\sum_{i\in \mathbb N} \mathbf 1_{\{\exists t\in [0,T]: X_i+\beta^{(i)}_t \in [a,a+1)\}}\right] < \infty.
\]
\end{lemma}

\begin{proof}[Proof of Theorem \ref{:I}\ref{:IF}]
	Suppose that $\mathbf X_0$ is bounded.
	In this case, from Proposition \ref{:IF_fixed_t} we have $\mathbb P(Z_t(\mathbb R)< \infty) = 1$ for every $t>0$. 
	Noticing that almost surely $t\mapsto Z_t(\mathbb R)$ is a non-increasing function on $(0,\infty)$ due to the coalescing mechanism, therefore $\mathbb P(Z_t(\mathbb R)<\infty, \forall t>0) = 1$.
	Also noticing that $Z_t(U)\leq Z_t(\mathbb R)$ for every $t>0$ almost surely, we have the desired result in this case.
	
	Now  we consider the case when $\mathbf X_0$ is unbounded.
	In this case, it is suffice to show that
\begin{equation} \label{:IF.0}
	\mathbb E \left[\sup_{t\in [\delta, T]} Z_t(U)\right]<\infty 
\end{equation}
	for every $0<\delta < T< \infty$.
	Let us fix arbitrary $0<\delta < T<\infty$.
	Notice that $t\mapsto Z_t(U)$ is not a non-increasing process in general, because particles may enter $U$ through ts boundary. 
	But we can still control the population in the region $U$ by considering how many particles pass though its boundary, thanks to the fact that there is no creation of new particles, and all particles move continuously.
	More precisely, we have the following inequality
\begin{equation} \label{:IF.1}
	\sup_{t\in [\delta, T]} Z_t(U) 
	\leq Z_\delta(U) + \sum_{i\in I_\delta} \mathbf 1_{\{\exists t\in [\delta, T]\cap [0,\zeta_i): B^{(i)}_t \in \partial U\}}, \quad \text{a.s.} 
\end{equation}
	Here the second term on the right hand side of \eqref{:IF.1} is the number of particles that ever crossed the boundary of $U$ in the time interval $[\delta, T]$.
	Since $U$ is an interval, its boundary set $\partial U := \overline U \setminus U$ consists of at most two points. 
	In particular, there exists $-\infty < a_0< a_1< \infty$ such that $\partial U \subset \{a_0 + 1/2,a_1+1/2\}$.
	Therefore we have	
\begin{equation}\label{:IF.21}
	\sup_{t\in [\delta, T]} Z_t(U) 
	\leq Z_\delta(U) + \sum_{k=0,1}\sum_{i\in I_\delta} \mathbf 1_{\{\exists t\in [\delta, T]: B^{(i)}_t \in [a_k,a_k+1)\}}, \quad \text{a.s.} 
\end{equation}

	From Theorem \ref{:I} \ref{:II}, we know that $I_\delta$, the set of labels of the living particles at time $\delta$, has (countable) infinitely many elements. 
	We claim that \eqref{:BS.1} holds with the random sequence $(X_i)_{i\in \mathbb N}$ being replaced by $(B^{(i)}_\delta)_{i\in I_\delta}$. 
	In fact, fixing arbitrary $0<2\varepsilon < \gamma < 1$ and $a\in \mathbb R$, by Proposition \ref{:IF_fixed_t}, 
	\[\mathbb E[Z_\delta((a-1,a+1))] \leq  \frac{1}{1 - \gamma}\left[\int_{a-1}^{a+1}v^{(\Lambda, \mu)}_{\delta, y} \mathrm d y + \frac{1}{2}\mathcal V^{(\Lambda,\mu,\mathbb R)}_\delta \right] + 2\Cr{:IFP.31}((a-1,a+1), \mathbb R, \delta, \gamma).\]
	Notice that, by \eqref{:II.522} and \eqref{:II.523},  
\[
	\int_{a-1}^{a+1} v^{( \Lambda,\mu)}_{\delta,y}\mathrm dy 
	\leq \int_{a-1}^{a+1} v^{(\mathbb R, \mathbf 0)}_{\delta,y}\mathrm dy=\frac{4}{\delta},
\]
	and that $\mathcal V^{(\Lambda, \mu, \mathbb R)}_\delta= 0$. 
	It is also clear from the definition of $\Cr{:IFP.31}$ in Lemma \ref{:IFP_3} that
\[
	\Cr{:IFP.31}((a-1,a+1), \mathbb R, \delta, \gamma) = \Cr{:IFP.31}((-1,1), \mathbb R, \delta, \gamma).
\]
	Therefore we have
\[
	\mathbb E[Z_\delta((a-1,a+1))] 
	\leq \frac{4}{(1-\gamma)\delta} + 2\Cr{:IFP.31}((-1,1), \mathbb R, \delta,\gamma) .
\]
	Since $a\in \mathbb R$ is arbitrary, we have
\[
	\sup_{a\in \mathbb R} \mathbb E\Big[\sum_{i\in I_\delta} \mathbf 1_{\{B^{(i)}_\delta\in [a,a+1)\}}\Big] 
	=\sup_{a\in \mathbb R} \mathbb E[Z_\delta([a,a+1))] < \infty
\]
as claimed.

	Define $\beta^{(i)}_t := B_{\delta+t}^{(i)} -B_\delta^{(i)}$ for every $i\in \mathbb N$ and $t\geq 0$. 
	From the Markov property of the Brownian motions, we have that $\{(\beta^{(i)}_t)_{t\geq 0}:i\in \mathbb N\}$ is a family of independent Brownian motions which are independent of $(B_\delta^{(i)})_{i\in I_\delta}$. 
	Now we can apply Lemma \ref{:BS}, and obtain that
\[
	\sup_{a\in \mathbb R} \mathbb E\left[\sum_{i\in I_\delta} \mathbf 1_{\{\exists t\in [\delta,T]: B_t^{(i)} \in [a,a+1)\}} \right]
	< \infty.
\]
	This says that the second term on the right hand side of \eqref{:IF.21} has finite expectation. 
	We already know from Proposition \ref{:IF_fixed_t} that the expectation of the first term on the right hand side of \eqref{:IF.21} is finite (see Remark \ref{thm:Finitem}).
	Therefore \eqref{:IF.0} holds as desired. 
\end{proof}
\end{note}

	For any open subset $G$ of $[0,\infty) \times \mathbb R$, let us define random variable
\[
	Z(G) := \sum_{i\in I_0} \mathbf 1_{\{\exists t>0 \text{~s.t.~} (t,X^{(i)}_t) \in G\}}
\]
	which is the number of particles whose time-space trajectory intersects with  $G$.
	
\begin{proposition} \label{prop:finitecrossing}
	Let $U$ be an open interval such that $U \cap \supp(\Lambda,\mu)$ is bounded.
	Let $0< \delta < T < \infty$ be arbitrary.
	Then $
	\mathbb E[Z((\delta,T)\times U )]< \infty.$
\end{proposition}
\begin{proof}
	Let us consider covering the space $\mathbb R$ by open intervals $L_j := (j-1,j+1), j\in \mathbb N$ and show that $\mathbb E[Z_\delta(L_j)]$ is uniformly bounded in $j$. 
	In fact, fixing arbitrary $0< \gamma < 1$ and $j\in \mathbb N$, by Proposition \ref{:IF_fixed_t}, we have
	\[\mathbb E[Z_\delta(L_j)] \leq  \frac{1}{1 - \gamma}\left[\int_{j-1}^{j+1}v^{(\Lambda, \mu)}_{\delta, y} \mathrm d y + \frac{1}{2}\mathcal V^{(\Lambda,\mu,\mathbb R)}_\delta \right] + 2\Cr{:IFP.31}(L_j, \mathbb R, \delta, \gamma).\]
	Notice that, by \eqref{:II.522} and \eqref{:II.523},  
	\[
	\int_{j-1}^{j+1} v^{( \Lambda,\mu)}_{\delta,y}\mathrm dy 
	\leq \int_{j-1}^{j+1} v^{(\mathbb R, \mathbf 0)}_{\delta,y}\mathrm dy=\frac{4}{\delta},
	\]
	and that $\mathcal V^{(\Lambda, \mu, \mathbb R)}_\delta= 0$. 
	It is also clear from the definition of $\Cr{:IFP.31}$ in Lemma \ref{:IFP_3} that
	\[
	\Cr{:IFP.31}(L_j, \mathbb R, \delta, \gamma) = \Cr{:IFP.31}(L_0, \mathbb R, \delta, \gamma).
	\]
	Therefore we have
	\[
	\mathbb E[Z_\delta(L_j)] 
	\leq \frac{4}{(1-\gamma)\delta} + 2\Cr{:IFP.31}(L_0, \mathbb R, \delta,\gamma) 
	=: \C\label{:Cgamma}(\delta,\gamma).
	\]
	
	Define 
	$\beta_t^{(i)} := B_{\delta + t}^{(i)} - B_{\delta}^{(i)}$ for every $i \in \N, t \geq 0.$ From the Markov property of Brownian motions, $\{(\beta_t^{(i)})_{t \geq 0} : i \in \N\}$ is a family of independent Brownian motions which are independent of  $(B_\delta^{(i)})_{i\in \mathbb N}$. 
	Therefore, 
	\[N_j := Z_{\delta}(L_j) = \#\{i\in \mathbb N :  \zeta_i > \delta, B_\delta^{(i)} \in L_j\}, \quad j\in \mathbb N,\] 
	is independent from 
	\[
	M_{T- \delta}^{(i)} := \sup_{t \in [0,T- \delta]}|\beta_t^{(i)}|, \quad i\in \mathbb N.
	\] 
	One can verify the following inequality
	\begin{align}
		& Z((\delta,T)\times U) \leq Z_\delta(U) + \sum_{j \in \N: L_j  \not\subset U} \sum_{i\in \mathbb N:\delta > \zeta_i, B_\delta^{(i)} \in L_j} \mathbf 1_{\{M_{T - \delta}^{(i)} \geq \dist(L_j, U) \}}.
	\end{align}
	Here $\dist(L_j, U)$ denotes the distance between sets $L_j$ and $U$, which is also the distance a Brownian particle at least needs to travel from $L_j$ to $U$.
	Taking expectations, we have
	\begin{align}
		&\mathbb E\left[Z((\delta,T)\times U) \right] 
		\leq \mathbb E[Z_\delta(U)] + \sum_{j\in \N:  L_j  \not\subset U} \mathbb E\left[ \sum_{i\in \mathbb N:\delta > \zeta_i, B_\delta^{(i)} \in L_j} \mathbf 1_{\{M_{T - \delta}^{(i)} \geq \dist(L_j, U) \}}\right].
	\end{align}
	It follows from Proposition \ref{:IF_fixed_t} and Remark \ref{thm:Finitem} that  the first expectation on the right hand side is finite. 
	Now from the independence of  $(N_j)_{j \in \N}$ with the collection of the i.i.d. random variables $(M_{T - \delta}^{(i)})_{i \in \N}$, we can apply Wald's  lemma to calculate the second expectation on the right hand side. 
	Thus
	\begin{align*}
		&\mathbb E\left[Z((\delta,T)\times U)\right] 
		\leq \mathbb E[Z_\delta(U)] + \sum_{j \in \N: L_j\not\subset U} \mathbb E[N_j]\cdot \mathbb P\left(M_{T - \delta}^{(0)} \geq \dist(L_j,U) \right)\\
		&\leq \mathbb E[Z_\delta(U)] + \Cr{:Cgamma}(\delta, \gamma) \sum_{j \in \N: L_j\not\subset U} \Cr{:Ctemp}(T-\delta) \exp\left(-\frac{\dist(L_j,U)^2}{2(T-\delta)}\right)< \infty
	\end{align*}
	as desired.
	Here, $\C\label{:Ctemp}(T-\delta)$ is a finite constant only depending on $T-\delta$.
\end{proof}

\begin{proof}[Proof of Theorem \ref{:I}\ref{:IF}]
Clearly  $Z_t(U) \leq Z((\delta, T)\times U)$,  for every $t\in (\delta, T)$, almost surely. Thus, it
	suffices to show that $\mathbb P(Z((\delta, T)\times U) < \infty) = 1$
for  arbitrarily fixed $0<\delta < T$. 
Now the desired result follows from Proposition \ref{prop:finitecrossing}.
\end{proof}
	Let us now give the proofs of Lemmas \ref{:IFP_1}--\ref{:IFP_3}.
\begin{proof}[Proof of Lemma \ref{:IFP_1}]
	Noticing that for every $n\in \mathbb N$,  
	$\mathbf X^{(n)} := \{X^{(i)}_t: t\geq 0, i = 1,\dots, n\}$ is a coalescing Brownian particle system with initial configuration $(x_i)_{i=1}^n$. 
	Denote by $I_t^{(n)}$ the random set of the indexes of the particles alive at a given time $t\geq 0$ in the particle system $\mathbf X^{(n)}$.
	From Shiga's duality \cite{MR948717}*{Theorem 5.2},
\[
	\ex \left[\prod_{i \in I_t^{(n)}}\left(1 - u_{0, X_t^{(i)}}\right)\right] 
	= \mathbf E_{\varepsilon \mathbf 1_U} \left[\prod_{i = 1}^n\left(1 - u_{t, x_i}\right)\right], \quad n\in \mathbb N.
\]
	Taking $n\uparrow\infty$, we get by monotone convergence theorem that
\[
	\mathbb E\left[\prod_{i\in I_t}\left(1-u_{0,X_t^{(i)}}\right)\right] 
	= \mathbf E_{\varepsilon \mathbf 1_U}\left[\prod_{i=1}^\infty (1-u_{t,x_i})\right].
\]
	Since $u_{0,x} = \varepsilon\mathbf 1_U(x)$, we can rewrite the above as  
\begin{align}
	&\ex \left[(1 - \ve)^{Z_t(U)}\right]= \mathbf E_{\varepsilon \mathbf 1_U}\left[\exp\left(\sum_{i = 1}^\infty \log(1 - u_{t, x_i})\right)\right]. 
\end{align}

	On one hand, from the fact that $\log (1-z) \leq -z$ for every $z\in [0,1]$, we have \eqref{:IFP_12} holds as desired.  
	  On the other hand, from  the fact that $z \mapsto \theta(z) =-\log(1-z)/z$ is a positive increasing function on $(0, 1)$, we can verify that $-\log (1-z) \leq \theta(\gamma) z$ for every $z\in [0,\gamma]$. 
	Note that almost surely on the event $\{\sup_{s \leq t, y \in F} u_{s, y} \leq \gamma\}$, for every $i\in \mathbb N$, we have $u_{t,x_i}\in [0,\gamma]$ and therefore $-\log(1-u_{t,x_i}) \leq \theta(\gamma) u_{t,x_i}$.
	So we can verify that
\begin{align}
	&\ex \left[(1 - \ve)^{Z_t(U)}\right]
	\geq \mathbf E_{\varepsilon \mathbf 1_U}\left[\exp\left(-\theta(\gamma)\sum_{i = 1}^\infty u_{t, x_i}\right)\mathbf 1_{\{\sup_{s \leq t, y \in F} u_{s, y} \leq \gamma\}}\right]
	\\&\geq \mathbf E_{\varepsilon \mathbf 1_U} \left[\exp\left(-\theta(\gamma)\sum_{i = 1}^\infty u_{t, x_i}\right)\right] - \mathbf P_{\varepsilon \mathbf 1_U }\left(\sup_{s \leq t, y \in F} u_{s, y} > \gamma\right).
	\qedhere
\end{align}
\end{proof}

\begin{proof}[Proof of Lemma \ref{:IFP_2}]
	Notice that, for the statement (2), we only have to show that $\mathcal V_t^{(\Lambda,\mu,F)}$ is finite when $F$ is the \emph{smallest} closed interval containing $\cup_{i\in \mathbb N}(x_i-1,x_i+1)$. 
	So let us denote by $F$ this smallest closed interval, and  by $\tilde F$ the smallest closed interval containing  $\{x_i:i\in \mathbb N\}.$
	From the condition  that $U \cap \{x_i:i\in \mathbb N\}$ is bounded, it is easy to see that $U \cap \tilde F$ is bounded.
	It is also clear that $(\Lambda, \mu) \preceq (\tilde F,\mathbf 0)$.	
	There are four cases to consider, for which we derive both statements (1) and (2).

	\emph{Case (1), $\tilde F = \mathbb R$}: In this case $U$ must be a bounded interval, say $(\alpha,\beta)$. From \eqref{:II.523} we know that $\int_U v^{(\Lambda,\mu)}_{t,x} \mathrm dx \leq 2(\beta - \alpha)/t$ is finite. 
	It is also obvious that $\mathcal V^{(\Lambda,\mu,\mathbb R)}_t = 0$ is finite.

	\emph{Case (2), $\tilde F = (-\infty, b]$ for some $b\in \mathbb R$}:
	In this case, $F=(-\infty,b+1]$ and $U$ must be the subset of $(\alpha,\infty)$ for some $\alpha \in \mathbb R$.
	From \eqref{:II.522} and \eqref{:II.524}, we have
\[
	\int_U v^{(\Lambda,\mu)}_{t,x} \mathrm dx 
	\leq \int_\alpha^\infty \left( v_{t,x}^{(\mathbb R,\mathbf 0)}\mathbf 1_{x<b} + v_{t,x-b}^{((-\infty, 0], \mathbf 0)}\mathbf 1_{x\geq b} \right)\mathrm dx.
\]
	One can verify that the integral on the right hand side is finite using \eqref{:II.523} and \eqref{:II.526}.
	One can also verify that 
\[
	\mathcal V^{(\Lambda, \mu, F)}_t
	\leq \int_0^t  \int_{b+1}^{\infty}  \left(v^{((-\infty,0],\mathbf 0)}_{r,z-b}\right)^2 \, \mathrm dz \, \mathrm dr
\]
	where the right hand side is finite by \eqref{:II.526}.

	\emph{Case (3), $\tilde F = [a,\infty)$ for some $a\in\mathbb R$}:
	This is similar to Case 2, thanks to the spatial symmetry  of the PDE \eqref{:.05}. 

	\emph{Case (4), $\tilde F=[a,b] $ for some $-\infty< a< b<\infty$}:
	Now $F=[a-1,b+1]$.
	From \eqref{:II.522} and \eqref{:II.524}, we have
\[
	\int_U v^{(\Lambda,\mu)}_{t,x} \mathrm dx 
	\leq \int v_{t,x-\frac{a+b}{2}}^{([-\frac{b-a}{2},\frac{b-a}{2}],\mathbf 0)}\mathrm dx.
\]
	One can verify that the integral on the right hand side is finite using \eqref{:II.54}.
	One can also verify that 
\[
	\mathcal V^{(\Lambda, \mu, F)}_t
	\leq \int_0^t  \left(\int_{-\infty}^{a-1} + \int_{b+1}^{\infty}\right)  \left(v_{r,z-\frac{a+b}{2}}^{([-\frac{b-a}{2},\frac{b-a}{2}],\mathbf 0)}\right)^2 \mathrm dz \mathrm dr
\]
where the right hand side is finite by \eqref{:II.54}. 
\begin{note}
	Let us verify those here.
	In fact, there exists a constant $\C\label{:II.545}>0$ such that 
\begin{align}
	&\int_k^\infty v^{([-k,k],\mathbf 0)}_{t,x}\mathrm dx \leq \int_k^\infty \frac{\Cr{:II.53}}{t}\Big(1+\frac{x-k}{\sqrt{t}}\Big)e^{-\frac{|x-k|^2}{2t}}\mathrm dx
	\\&=\int_0^\infty \frac{\Cr{:II.53}}{t}\Big(1+\frac{z}{\sqrt{t}}\Big)e^{-\frac{z^2}{2t}}\mathrm dz
	=\frac{\Cr{:II.53}}{\sqrt{t}}\int_0^\infty (1+w)e^{-\frac{w^2}{2}}\mathrm dw = \frac{\Cr{:II.545}}{\sqrt{t}}.
\end{align}
	Therefore, from \eqref{:II.52} and the fact that $v^{([-k,k],\mathbf 0)}_{t,x}$ is an even function in $x$, we have
\begin{equation}\label{:II.445}
	\int v^{([-k,k],\mathbf 0)}_{t,x}\mathrm dx \leq \frac{4k}{t}+\frac{2\Cr{:II.545}}{\sqrt{t}},\quad t>0.
\end{equation}
	We can also verify that 
\begin{align}
	&\int_n^\infty (v^{([-k,k],\mathbf 0)}_{t,x})^2 \mathrm dx \leq \int_n^\infty \frac{\Cr{:II.53}^2}{t^2} \Big(1+ \frac{x-k}{\sqrt{t}}\Big)^2 e^{-\frac{(x-k)^2}{t}} \mathrm dx
	\\&\leq \frac{\Cr{:II.53}^2}{t^2} \int_{n-k}^\infty  \Big(1+ \frac{x}{\sqrt{t}}\Big)^2 e^{-\frac{x^2}{t}} \mathrm dx \leq  \frac{2\Cr{:II.53}^2}{t^2}\int_{n-k}^\infty  \Big(1+ \frac{x^2}{t}\Big) e^{-\frac{x^2}{t}} \mathrm dx
	\\&= \frac{2\Cr{:II.53}^2 \sqrt{t}}{t^2}\Big( \int_{\frac{n-k}{\sqrt{t}}}^\infty   e^{-x^2} \mathrm dx + \int_{\frac{n-k}{\sqrt{t}}}^\infty  x^2 e^{-x^2} \mathrm dx\Big).
\end{align}
	According to the internet, we have
\begin{equation}
	\int_{z}^\infty e^{-x^2}\mathrm dx
	= \frac{1}{2}\sqrt{\pi}\operatorname{erfc}(z)
	\leq \frac{1}{2}\sqrt{\pi}e^{-z^2}, \quad z> 0
\end{equation}
	and
\begin{equation}
	\int_{z}^\infty x^2e^{-x^2} dx 
	= \frac{1}{4}(\sqrt{\pi} \operatorname{erfc}(z)+2e^{-z^2}z) \leq \frac{1}{4}(\sqrt{\pi} +2z)e^{-z^2}, 
	\quad z>0.
\end{equation}
	Putting those back, there exists a constant $\C\label{:II.546}$, independent of $k>0$, such that  
\begin{align}
	&\int_n^\infty (v^{([-k,k],\mathbf 0)}_{t,x})^2 \mathrm dx \leq \frac{2\Cr{:II.53}^2 \sqrt{t}}{t^2}\Big( \frac{3}{4}\sqrt{\pi}e^{-\frac{(n-k)^2}{t}}+ \frac{1}{2}\frac{n-k}{\sqrt{t}}e^{-\frac{(n-k)^2}{t}}\Big)
	\\&\leq \Cr{:II.546}t^{-\frac{3}{2}}\Big(1+\frac{n-k}{\sqrt{t}}\Big)e^{-\frac{(n-k)^2}{t}}, \quad t>0, n>k.
\end{align}
	Therefore,
\begin{equation} 
	\int_0^t \mathrm dr \int_{(-n,n)^{\mathrm c}} (v^{([-k,k],\mathbf 0)}_{t,x})^2\mathrm dz
	\leq 2\Cr{:II.546}\int_0^t r^{-\frac{3}{2}}\Big(1+\frac{n-k}{\sqrt{r}}\Big)e^{-\frac{(n-k)^2}{r}}\mathrm dr, 
	\quad t>0, n>k.
\end{equation}
\end{note}
\end{proof}
\begin{note}
	Let us give a useful result for our SPDE.
	Write $G_{s,y;r,z} = G_{r-s,z-y}$ for every $r>s\geq 0$ and $y,z\in \mathbb R$.
	Let $\mathcal C_0(\mathbb R)^*$ be the space of signed Radon measures on $\mathbb R$ with finite total variation.
	By the Riesz-Markov representation theorem, $\mathcal C_0(\mathbb R)^*$ is also the space of continuous linear functional of $\mathcal C_0(\mathbb R)$. 
\begin{lemma} \label{:IFLI}
	Let $t>0$, $\mu\in \mathcal C_0(\mathbb R)^*$ and $F\in L^1([0,t]\times \mathbb R)$.
	Suppose that for every $s\in [0,t)$ and $y\in \mathbb R$,
\[
	\phi_{s,y}
	:= \int G_{s,y;t,x} \mu(\mathrm dx) - \iint_s^t G_{s,y;r,z}F_{r,z}\mathrm dr\mathrm dz \in \mathbb R.
\]
	Then $\phi \in L^2([0,t)\times \mathbb R)$ and $\phi_{s,\cdot}\in L^1(\mathbb R)$ for every $s\in [0,t)$. 
	Moreover, almost surely for every $r\in [0,t]$,
\begin{align}
	&\int  u_{r,x} \tilde \phi_r(\mathrm dx) - \int \phi_{0,x}u_{0,x}\mathrm dx
	\\&\label{:II.5431}=\iint_0^r F_{s,y}u_{s,y}\mathrm ds\mathrm dy+\iint_0^r \phi_{s,y} \sqrt{u_{s,y}(1-u_{s,y})}W(\mathrm ds\mathrm dy) 
\end{align}
	where $\tilde \phi_t :=\mu$, and for every $r\in [0,t)$, $\tilde \phi_r\in \mathcal C_0(\mathbb R)^*$ is defined so that $ \int f_x\tilde \phi_r(\mathrm dx) = \int f_x\phi_{r,x}\mathrm dx$ for every $f\in \mathcal C_0(\mathbb R)$.
\end{lemma}
	To proof the above lemma, we recall the stochastic Fubini's theorem.
\begin{lemma}[Stochastic Fubini's theorem, c.f. {\cite[Lemma 2.4]{MR863723}}]
	Suppose that $\Omega$ is a filtered probability space satisfying the usual condition.
	Denote by $\mathscr L$ the predictable $\sigma$-algebra of $\Omega \times [0,\infty)$. 
	Suppose that $(U,\mathscr U,m)$ is a finite measure space.
	Let $W$ be a space-time white noise living in $\Omega$.
	Suppose that $\psi$ is an $\mathscr L\otimes \mathscr B(\mathbb R)\otimes \mathscr U$-measurable map from $\Omega\times [0,\infty)\times \mathbb R \times U$ to $\mathbb R$ such that
\[
	\int_Um(\mathrm du)\iint_0^T \psi(\omega,t,x,u)^2 \mathrm dt\mathrm dx<\infty, \quad T>0, 
	\quad\text{a.s.}
\]
	Then there exists an $\mathscr L\otimes \mathscr U$-measurable map $\Psi$ from $\Omega \times [0,\infty)\times U$ to $\mathbb R$ such that for $m$-almost every $u\in U$,
\[
	\Psi(\omega,T,u)
	=\iint_0^T \psi(\omega,t,x,u)W(\mathrm dt,\mathrm dx), 
	\quad T>0, \quad \text{a.s.}
\]
	Furthermore, for any such $\Psi$, it holds that 
\[
	\int_U\Psi(\omega,T,u)m(\mathrm du)
	=\iint_0^T\Big(\int_U \psi(\omega,t,x,u)m(\mathrm du)\Big)W(\mathrm dt\mathrm dx),\quad T>0,
	\quad \text{a.s.}
\]
\end{lemma}
\begin{proof}[Proof of Lemma \ref{:IFLI}]
\emph{Step 1.}
	Let us verify that  $\phi \in L^2([0,t)\times \mathbb R)$ and $\phi_{s,\cdot}\in L^1(\mathbb R)$ for every $s\in [0,t)$.
	Notice that 
\[
	|\phi_{s,y}| 
	\leq \int G_{s,y;t,x}|\mu|(\mathrm dx) + \iint_s^t G_{s,y;r,z}|F_{r,z}|\mathrm dr\mathrm dz, 
	\quad s\in [0,t),y\in \mathbb R.
\]
	Therefore,
\[
	\int |\phi_{s,y}|\mathrm dy 
	\leq \int |\mu|(\mathrm dx) + \iint_s^t |F_{r,z}|\mathrm dr\mathrm dz <\infty, 
	\quad s\in [0,t).
\]
	Also notice that for every $s\in [0,t)$ and $y\in \mathbb R$,
\begin{align}
	&|\phi_{s,y}|^2 
	\leq 2 \Big(\int G_{s,y;t,x}|\mu|(\mathrm dx)\Big)^2 + 2\Big(\iint_s^t G_{s,y;r,z}|F_{r,z}|\mathrm dr\mathrm dz\Big)^2
	\\&\leq 2 |\mu|(\mathbb R)\int G_{s,y;t,x}^2 |\mu|(\mathrm dx) + 2\|F\|_{L^1([s,t]\times \mathbb R)}\iint_s^t G^2_{s,y;r,z}|F_{r,z}|\mathrm dr\mathrm dz
\end{align}
	We can verify for every $r>0$ and $z\in \mathbb R$,
\begin{equation}
	\iint_0^r G_{s,y;r,z}^2 \mathrm ds\mathrm dy
	\leq \iint_0^r \frac{G_{s,y;r,z}}{\sqrt{2\pi(t-s)}} \mathrm ds\mathrm dy
	= \frac{1}{\sqrt{2\pi}}\int_0^r \frac{1}{\sqrt{s}} \mathrm ds
	= \sqrt{\frac{2r}{\pi}}.
\end{equation}
	Therefore
\begin{align}
	&\iint_0^t |\phi_{s,y}|^2 \mathrm ds\mathrm dy \leq 2 \sqrt{\frac{2t}{\pi}} |\mu|(\mathbb R)^2+2 \|F\|_{L^1([s,t]\times \mathbb R)} \iint_0^t \sqrt{\frac{2r}{\pi}}|F_{r,z}|\mathrm dr\mathrm dz
	\\& \leq 2 \sqrt{\frac{2t}{\pi}} \Big(|\mu|(\mathbb R)^2+ \|F\|_{L^1([0,t]\times \mathbb R)}^2\Big). 
\end{align}
\emph{Step 2.} 
	Let us verify that \eqref{:II.5431} holds almost surely when $r=t$.
	To simplify the notation, let us write the ``measure"
\[
	M^u(\mathrm ds\mathrm dy) 
	:= u_{0,y}\delta_0(\mathrm ds)\mathrm dy + \sqrt{u_{s,y}(1-u_{s,y})}W(\mathrm ds\mathrm dy)
\]
	whose integral is defined as a sum of the classical integral and Walsh's stochastic integral in an obvious way. 
	From \eqref{:II.35}, we have that
\[
	u_{r,z} 
	= \iint_0^r G_{s,y;r,z}M^u(\mathrm ds\mathrm dy), \quad \text{a.s.} \quad r> 0, z\in \mathbb R.
\]
	We will explain that almost surely
\begin{align}
	&\int  u_{t,x}\tilde \phi_t(\mathrm dx) 
	= \int \tilde \phi_{t}(\mathrm dx) \iint_0^t G_{s,y;t,x} M^u(\mathrm ds\mathrm dy) 
	\\&\label{:II.412}= \iint_0^tM^u(\mathrm ds\mathrm dy)\int G_{s,y;t,x}\tilde \phi_{t} (\mathrm dx)   
	\\&=\iint_0^t \phi_{s,y} M^u(\mathrm ds\mathrm dy)+ \iint_0^t M^u(\mathrm ds\mathrm dy)\iint_s^t   G_{s,y;r,z}F_{r,z}\mathrm dr\mathrm dz  
	\\&\label{:II.413}=\iint_0^t \phi_{s,y}  M^u(\mathrm ds\mathrm dy)+ \iint_0^t F_{r,z}\mathrm dr\mathrm dz  \iint_0^rG_{s,y;r,z} M^u(\mathrm ds\mathrm dy)
	\\&=\int \phi_{0,x} u_{0,x}\mathrm dx +\iint_0^t \phi_{s,y}\sqrt{u_{s,y}(1-u_{s,y})}W(\mathrm ds\mathrm dy)+ \iint_0^t F_{s,y} u_{s,y}\mathrm ds\mathrm dy. 
\end{align}
	To interchange the order of the integral in \eqref{:II.412} using (stochastic) Fubini's theorem, we verify that for every $t>0$, 
\[
	\int  |\tilde \phi_{t}|(\mathrm dx) \iint_0^t  G_{s,y;t,x} \delta_0(\mathrm ds)\mathrm dy 
	= \int |\tilde \phi_{t}|(\mathrm dx)<\infty 
\]
	and 
\begin{equation}
	\int |\tilde \phi_{t}|(\mathrm dx) \iint_0^t G_{s,y;t,x}^2 \sigma(u_{s,y})^2\mathrm ds\mathrm dy
	\leq \sqrt{\frac{2t}{\pi}}\int |\tilde \phi_{t}| (\mathrm dx) <\infty.
\end{equation}
	Here $\sigma(h) :=\sqrt{h(1-h)}$ for every $h\in [0,1]$.
	To interchange the order of the integral in the second term of \eqref{:II.413} using (stochastic) Fubini's theorem, we verify that for every $t>0$, 
\[
	\iint_0^t |F_{r,z}|\mathrm dr\mathrm dz  \iint_0^rG_{s,y;r,z} \delta_0(\mathrm ds)\mathrm dy
	=\iint_0^t |F_{r,z}|\mathrm dr\mathrm dz  <\infty
\]
	and
\[
	\iint_0^t |F_{r,z}|\mathrm dr\mathrm dz  \iint_0^rG_{s,y;r,z}^2 \sigma(u_{s,y})^2 \mathrm ds\mathrm dy
	\leq \iint_0^t |F_{r,z}|\sqrt{\frac{2r}{\pi}}\mathrm dr\mathrm dz<\infty. 
\]
	
\emph{Step 3.}
	Let us verify that \eqref{:II.5431} holds almost surely for every $r\in [0,t]$.
	This is trivial when $r=0$, and already been shown when $r = t$ in Step 2.
	Now let $r\in (0,t)$ be arbitrary. For every $s\in [0,r)$, we can verify from Fubini's theorem that
\begin{align}
	&\int G_{s,y;r,z} \phi_{r,z}\mathrm dz 
	= \int G_{s,y;r,z} \mathrm dz\Big(\int G_{r,z;t,x} \tilde\phi_{t}(\mathrm dx) - \iint_r^t G_{r,z;\gamma,\alpha}F_{\gamma,\alpha}\mathrm d\gamma\mathrm d\alpha\Big)
	\\&= \int G_{s,y;t,x} \tilde\phi_{t}(\mathrm dx) -  \iint_r^t G_{s,y;\gamma,\alpha}F_{\gamma,\alpha}\mathrm d\gamma\mathrm d\alpha.
\end{align}
	This implies that, for every $s\in [0,r)$ and $y\in \mathbb R$,
\begin{align}
	\phi_{s,y}
	= \int G_{s,y;r,z} \phi_{r,z}\mathrm dz - \iint_s^r G_{s,y;\gamma,\alpha}F_{\gamma,\alpha}\mathrm d\gamma\mathrm d\alpha.
\end{align}
	Now from what we have proved in Step 2, we can get the desired result for this step.

\emph{Final Step.}
	The previous step says that \eqref{:II.5431} holds almost surely for every $r\in [0,t]$. 
	Let us now verify that \eqref{:II.5431} holds for every $r\in [0,t]$ almost surely. 
	Thanks to Step 3, and the fact that the continuous modification of a continuous process is indistinguishable to that process, we only have to verify that almost surely, the left hand side of \eqref{:II.5431} is continuous in $r\in [0,t]$.
	To do this, let $h\in \mathrm b\mathcal C([0,t]\times \mathbb R)$ be arbitrary. Then we can verify that for every $s\in [0,t)$,
\begin{align}
	&\int h_{s,y}\tilde \phi_s(\mathrm dy)=\int h_{s,y}\mathrm dy\Big(\int G_{s,y;t,x}\tilde \phi_{t}(\mathrm dx) - \iint_s^t G_{s,y;r,z}F_{r,z}\mathrm dr\mathrm dz\Big)  
	\\&=\int \tilde\phi_{t}(\mathrm dx)\int G_{s,y;t,x} h_{s,y}\mathrm dy- \iint_s^t F_{r,z}\mathrm dr\mathrm dz\int G_{s,y;r,z}h_{s,y}\mathrm dy  
	\\\label{:II.414}&=\int \Pi_{x}[h_{s,B_{t-s}}]\tilde \phi_t(\mathrm dx)- \iint_s^t F_{r,z} \Pi_{z}[h_{s,B_{r-s}}]  \mathrm dr\mathrm dz
\end{align}
	where $\Pi_{x}$ is the law of a Brownian motion $(B_t)_{t\geq 0}$ initiated at position $x\in \mathbb R$. 
	Obviously \eqref{:II.414} also holds when $s=t$. 
	It is then easy to see that $s\mapsto \int h_{s,y}\tilde \phi_s(\mathrm dy)$ is a continuous map on $[0,t]$ using bounded convergence theorem.
	Observing that $u\in \mathrm b\mathcal C([0,t]\times \mathbb R)$ almost surely, we are done.
\end{proof}
\end{note}
\begin{proof}[Proof of Lemma \ref{:IFP_25}]
	 Let $n\in \mathbb N$ be arbitrary. 
	Define a stopping time \[ \tau := \inf\left\{s\geq 0: \sup_{y \in F}u_{s,y} > \gamma\right\}, \] and a process
\begin{align} \label{:II.41}
	M_s^{(\gamma,n)}
	:=
\begin{cases} \displaystyle 
	\frac{1}{1-\gamma}\int u_{s,y}v^{(\varnothing,\mu^{(\gamma,n)})}_{t-s,y}\mathrm dy,& 
	\quad s\in [0,t), 
\\ \displaystyle
	\frac{1}{1-\gamma}\int u_{t,y}\mu^{(\gamma,n)}(\mathrm dy),& \quad s=t. 
\end{cases}
\end{align} 
	where
\(
	\mu^{(\gamma, n)}(\mathrm d y) 
	= (1 - \gamma)\theta(\gamma)\sum_{i = 1}^n\delta_{x_i}.
\)
	One can show that $\{M_s^{(\gamma, n)} : s\in [0,t]\}$ is a continuous semi-martingale. 
	In fact, using the stochastic Fubini's theorem for space-time white noise (c.f. \cite{MR863723}*{Lemma 2.4}), it is standard to verify that almost surely, for every $s\in [0,t]$, 
\begin{align}
	&M_s^{(\gamma,n)} - M_0^{(\gamma,n)} 
	\\&= \frac{1}{1-\gamma}\Big(\frac{1}{2}\iint_0^s (v^{(\varnothing,\mu^{(\gamma,n)})}_{t-r,z})^2  u_{r,z} \mathrm dr\mathrm dz + \iint_0^s v^{(\varnothing,\mu^{(\gamma,n)})}_{t-r,z} \sqrt{ u_{r,z}\big(1-u_{r,z}\big)}W(\mathrm dr\mathrm dz)\Big).
\end{align}
\begin{note}
	Let us explain the above claim here.
	Firstly, from \cite[Theorem 3]{MR700049} we have
\[
	\iint_0^\infty (v^{(\varnothing,\mu^{(\gamma,n)})}_{s,y})^2 \mathrm ds\mathrm dy	<\infty.
\]
	Secondly, from \cite[Proof of Theorem 3]{MR700049} we have
\[
	v^{(\varnothing,\mu^{(\gamma,n)})}_{t,x} 
	= \int G_{0,y;t,x}\mu^{(\gamma, n)}(\mathrm dy) - \frac{1}{2}\iint_0^t G_{s,y;t,x} (v^{(\varnothing,\mu^{(\gamma,n)})}_{s,y})^{2}\mathrm ds\mathrm dy, 
	\quad t>0,x\in \mathbb R.
\]
	From this, we can verify that 
\begin{align}
	&v^{(\varnothing,\mu^{(\gamma,n)})}_{t-s,y}= \int G_{0,x;t-s,y}\mu^{(\gamma, n)}(\mathrm dx) - \frac{1}{2}\iint_{0}^{t-s} G_{r,z;t-s,y} (v^{(\varnothing,\mu^{(\gamma,n)})}_{r,z})^{2}\mathrm dr\mathrm dz
	\\&= \int G_{s,y;t,x} \mu^{(\gamma, n)}(\mathrm dx) - \frac{1}{2}\iint_{s}^{t} G_{s,y;r,z}(v^{(\varnothing,\mu^{(\gamma,n)})}_{t-r,z})^2\mathrm dr\mathrm dz, \quad s\in [0,t),y\in \mathbb R.
\end{align}
	Now the desired claim follows from Lemma \ref{:IFLI}. 
\end{note}
	By Ito's formula, we can verify that almost surely for every $s\in [0,t]$, 
\begin{align}
	&e^{-M^{(\gamma,n)}_s} - e^{-M^{(\gamma,n)}_0} 
\\ \label{:II.42}
	&=  -\frac{1}{2(1-\gamma)^2}\iint_0^s e^{-M^{(\gamma, n)}_r} (v^{(\varnothing,\mu^{(\gamma,n)})}_{t-r,z})^2(u_{r,z}^2-\gamma u_{r,z})  \mathrm dr\mathrm dz 
      \\& \qquad - \frac{1}{1-\gamma}\iint_0^s e^{-M_r^{(\gamma, n)}}v^{(\varnothing,\mu^{(\gamma, n)})}_{t-r,z}\sqrt{u_{r,z}(1-u_{r,z})} W(\mathrm dr\mathrm dz). 
\end{align}
\begin{note}
\begin{align}
	&M_s^{(\gamma,n)} - M_0^{(\gamma,n)} 
	\\&= \frac{1}{1-\gamma}\Big(\frac{1}{2}\iint_0^s (v^{(\varnothing,\mu^{(\gamma,n)})}_{t-r,z})^2  u_{r,z} \mathrm dr\mathrm dz + \iint_0^s v^{(\varnothing,\mu^{(\gamma,n)})}_{t-r,z} \sqrt{ u_{r,z}\big(1-u_{r,z}\big)}W(\mathrm dr\mathrm dz)\Big).
\end{align}
	In fact, by Ito's formula, we have almost surely for every $s\in [0,t]$, 
\begin{align}
	&e^{-M_s^{(\gamma,n)}} - e^{-M_0^{(\gamma,n)}}= -\int_0^s e^{-M_r^{(\gamma,n)}} \mathrm dM^{(\gamma,n)}_r + \frac{1}{2}\int_0^s e^{-M_r^{(\gamma,n)}}\mathrm d\langle M^{(\gamma,n)}\rangle_r 
	\\ & = -\frac{1}{2(1-\gamma)}\iint_0^s e^{-M_r^{(\gamma,n)}} (v^{(\varnothing,\mu^{(\gamma,n)})}_{t-r,x})^2 u_{r,z} \mathrm dr\mathrm dz 
	\\& \qquad - \frac{1}{1-\gamma}\iint_0^s e^{-M_r^{(\gamma,n)}}v^{(\varnothing, \mu^{(\gamma,n)})}_{t-r,z} \sqrt{u_{r,z}\big(1-u_{r,z}\big)}W(\mathrm dr\mathrm dz) 
	\\ &\qquad\qquad +\frac{1}{2(1-\gamma)^2}\iint_0^se^{-M_r^{(\gamma,n)}} (v^{(\varnothing,\mu^{(\gamma,n)})}_{t-r,z})^2 u_{r,z}(1-u_{r,z})\mathrm  dr\mathrm dz. 
\end{align}
\end{note}

	  To show \eqref{:IFP_251}, let  us take the expectation on \eqref{:II.42}  while setting $s=\tau \wedge t$.
Notice that
\[
	\mathbf E_{\varepsilon \mathbf 1_U}\Big[\iint_0^{\tau\wedge t} e^{-M_r^{(\gamma,n)}}v^{(\varnothing,\mu^{(\gamma,n)})}_{t-r,z}\sqrt{u_{r,z}(1-u_{r,z})} W(\mathrm dr\mathrm dz)\Big]
	= 0,
\]
	and
\begin{equation}
	\iint_0^{\tau\wedge t} e^{-M^{(\gamma,n)}_r} (v^{(\varnothing,\mu^{(\gamma,n)})}_{t-r,z})^2(u_{r,z}^2-\gamma u_{r,z})  \mathrm dr\mathrm dz
	\leq (1-\gamma) \int_0^{t} \mathrm dr  \int_{F^\mathrm c}  (v^{(\varnothing,\mu^{(\gamma,n)})}_{t-r,z})^2u_{r,z}  \mathrm dz \quad \text{~a.s.~} 
\end{equation}
	The last inequality follows by the observation that $u^2_{r,z} - \gamma u_{r,z} \leq 0$ for $r\leq \tau$ and $z\in F$.
	By the above and noticing that from \eqref{:II.35} that $\mathbf E_{\varepsilon \mathbf 1_U}[u_{r,z}]\leq \varepsilon$ for every $r\geq 0, z\in \mathbb R$,
	we get
\begin{align}
	\mathbb E[e^{-M^{(\gamma,n)}_{\tau \wedge t}}] - e^{-M^{(\gamma,n)}_0} \geq - \frac{\varepsilon}{2(1-\gamma)} \int_0^{t} \mathrm dr  \int_{F^\mathrm c}  (v^{(\varnothing,\mu^{(\gamma,n)})}_{r,z})^2  \mathrm dz.
\end{align}
	Now we have
\begin{align}
	&\mathbb E[e^{-M^{(\gamma,n)}_{t}}] + \mathbb P(\tau \leq t) \geq \mathbb E[e^{-M^{(\gamma,n)}_{\tau\wedge t}};\tau > t] + \mathbb E[e^{-M^{(\gamma,n)}_{\tau\wedge t}};\tau \leq t]
	\\\label{:II.50}&\geq e^{-M^{(\gamma,n)}_0} - \frac{\varepsilon}{2(1-\gamma)} \int_0^{t} \mathrm dr   \int_{F^\mathrm c}  (v^{(\varnothing,\mu^{(\gamma,n)})}_{r,z})^2  \mathrm dz.
\end{align}
	We want to take the limit  in \eqref{:II.50}  as $n\to \infty$.
	Notice that the $\mathcal T$-sequence $((\varnothing, \mu^{(\gamma,n)}))_{n\in \mathbb N}$ is increasing with respect to the partial order $\preceq$.
	Recall the map $\eta$ given by \eqref{:Ieta}.
	Define $\mu^{(\gamma)} := \eta^{(\Lambda, (1-\gamma)\theta(\gamma) \mu)}$.
	From \eqref{:lambda} and \eqref{:mu}, it is straightforward to verify that $(\mu^{(\gamma,n)})_{n\in\mathbb N}$ converges to $\mu^{(\gamma)}$ m-weakly as $n\to \infty$.
	Therefore, from \eqref{:II.522} and \eqref{:II.521}, the increasing sequence $(v^{(\varnothing, \mu^{(\gamma,n)})}_{r,z})_{n\in \mathbb N}$ converges to $v_{r,z}^{(\Lambda, (1-\gamma)\theta(\gamma) \mu)}$ for every $r>0$ and $z\in \mathbb R$.
	From \eqref{:II.522} and the fact that $(1-\gamma)\theta(\gamma) \leq 1$, we have $v_{r,z}^{(\Lambda, (1-\gamma)\theta(\gamma) \mu)}\leq v_{r,z}^{(\Lambda, \mu)}$ for every $r>0$ and $z\in \mathbb R$.
	Taking the limit  in \eqref{:II.50}  as $n\uparrow \infty$, we can now verify from the monotone convergence theorem that
\begin{align}
	&M_t^{(\gamma, n)} 
	\longrightarrow \theta(\gamma)\sum_{i = 1}^\infty u_{t, x_i},
	\\\label{:Later}&M^{(\gamma, n)}_0 
      	\longrightarrow \frac{\ve}{1 - \gamma}\int_U v^{(\Lambda, (1- \gamma)\theta(\gamma)\mu)}_{t, y} \mathrm d y 
      	\leq \frac{\ve}{1 - \gamma}\int_U v^{(\Lambda, \mu)}_{t, y} \mathrm d y,
      	\\&\int_0^{t} \mathrm dr \int_{F^\mathrm c} (v^{(\varnothing,\mu^{(\gamma,n)})}_{r,z})^2  \mathrm dz 
      	\longrightarrow \mathcal V^{(\Lambda, (1-\gamma)\theta(\gamma) \mu,F)}_t 
      	\leq \mathcal V^{(\Lambda, \mu,F)}_t
\end{align}
	which implies the desired \eqref{:IFP_251}. 

	To show \eqref{:IFP_252}, let us take the expectation on \eqref{:II.42} while replacing $s$ by $t$ and the arbitrary $\gamma\in [0,1)$ by $0$.
	This gives us 
\begin{equation}\label{:IE_Star} 
	\mathbb E[e^{-M_t^{(0,n)}}] \leq e^{-M_0^{(0,n)}}.
\end{equation}
			Taking the arbitrary $n\uparrow \infty$, we can get the desired \eqref{:IFP_252} from \eqref{:Later} and \eqref{:IE_Star}.
\end{proof}
\begin{proof}[Proof of Lemma \ref{:IFP_3}]
	Take $\varepsilon < \gamma /2$ and assume that $(u_{t,x})_{t\geq 0, x\in \mathbb R}$ solves the Wright-Fisher SPDE \eqref{:II.3} with initial condition $u_{0,\cdot} = \varepsilon \mathbf 1_U$.
		From \eqref{:II.35}, we have that 
				\begin{equation} \label{:Delta}
		u_{s,y} \leq \ve + N_s(y), \quad s\geq 0,y\in \mathbb R.
	\end{equation}
			where 
\[
	N_s(y) 
	:= \iint_0^s G_{s-r,y-z} \sqrt{u_{r,z}(1-u_{r,z})}W(\mathrm dr\mathrm dz).
\]
	Since $\varepsilon < \gamma/2$, we see from \eqref{:Delta} that $u_{s, y}$ is less than $\gamma$ if $|N_s(y)|$ is less than $\gamma/2$.
	Hence, it suffices to show that there exists a constant   $\C\label{:IFP_33}(U,F,t,\gamma) > 0$,  independent of $\varepsilon \in (0,\gamma/2)$, such that
\begin{equation} \label{:IFP_30}
	\mathbf P_{\varepsilon \mathbf 1_U}\left(\sup_{s \leq t, y \in F}|N_s(y)| > \frac{\gamma}{2}\right) 
	\leq \ve  \Cr{:IFP_33}(U,F,t,\gamma). 
\end{equation}
	Since $F$ is a closed interval, there are four cases:
\begin{enumerate}
\item $F = [a,\infty)$ for some $a\in \mathbb R$; 
\item $F = (-\infty,b]$ for some $b\in \mathbb R$; 
\item $F = \mathbb R$; and 
\item $F=[a,b]$ for some $-\infty < a \leq b< \infty$.
\end{enumerate}

	\emph{Case (1)}: Since $U$ is an interval and $U\cap F$ is bounded, we have $U\subset (-\infty,\beta)$ for some $\beta \in \mathbb R$. 
	Now we can apply \cite{MR1339735}*{Lemma 3.1}, and get that there exists a constant $\C\label{:IFP.32}>0$, independent of $U,F,t,\gamma$ and $\varepsilon$, such that 
\begin{align}
	&\mathbf P_{\varepsilon \mathbf 1_U}\left(\sup_{s \leq t, y \in [a,\infty)}|N_s(y)| > \frac{\gamma}{2}\right) 
	\\& \label{:IFP_31}\leq \Cr{:IFP.32} \Big(\frac{\gamma}{2}\Big)^{-20} (t\vee t^{22})  \int \varepsilon \mathbf 1_U(x) \mathrm dx \int G_{t,x-z}\mathbf 1_{[a,\infty)}(z)\mathrm dz
	\\& \leq \varepsilon \frac{\Cr{:IFP.32}}{\sqrt{2\pi t}} \Big(\frac{\gamma}{2}\Big)^{-20} (t\vee t^{22})  \int_{-\infty}^\beta \mathrm dx \int_a^\infty G_{t,x-z} \mathrm dz<\infty
\end{align}
	which implies the desired result \eqref{:IFP_30} for this case. 

	\textit{Case (2)}: This case is similar to Case (1) due to the spatial symmetry of the SPDE \eqref{:II.3}.  

	\textit{Case (3)}: Since $U\cap F$ is bounded, we have $U = (\alpha, \beta)$ for some $-\infty < \alpha \leq \beta <\infty$.
	In this case, by \cite{MR1339735}*{Lemma 3.1}, \eqref{:IFP_31} still holds for arbitrary $a\in \mathbb R$.
	Taking $a\downarrow -\infty$, we get from monotone convergence theorem that
\begin{equation}
	\mathbf P_{\varepsilon \mathbf 1_U}\left(\sup_{s \leq t, y \in \mathbb R}|N_s(y)| > \frac{\gamma}{2}\right) 
	\leq \Cr{:IFP.32} \varepsilon \Big(\frac{\gamma}{2}\Big)^{-20} (t\vee t^{22})  (\beta - \alpha). 
\end{equation}
	Therefore, the desired result \eqref{:IFP_30} also holds for Case (3). 

	\textit{Case (4)}: In this case $F=[a,b]$ is bounded.
	Let $p>4$.
	From \cite{MR1271224}*{Lemma 6.2}, there exists a constant $\C\label{:D_6}>0$ such that for any $s,s'\geq 0$ and $y,y'\in \mathbb R$, 
\[
	\mathcal K_{s,y;s',y'} 
	:= \iint_0^\infty |G_{s'-r,y'-z}-G_{s-r,y-z}|^2 \mathrm dr\mathrm dz \leq \Cr{:D_6} \big(|y'-y| + \sqrt{|s'-s|}\big)
\]
	where $G:=0$ on $(-\infty,0)\times \mathbb R$ for convention.
	From \eqref{:II.35}, Burkholder-Davis-Gundy inequality \cite{MR4226142}*{Theorem 20.12}, and Jensen's inequality, there exists a constant $\C\label{:D.7}(p)>0$, depending only on $p$, such that for every $s,s'\geq 0$ and $y,y'\in \mathbb R$,
\begin{align}
	&\mathbf E_{\varepsilon \mathbf 1_U}[|N_{s'}(y')-N_{s}(y)|^{2p}] 
	\\&= \mathbf E_{\varepsilon \mathbf 1_U}\left[\left|\iint_0^\infty (G_{s'-r,y'-z}-G_{s-r,y-z})\sqrt{u_{r,z}(1-u_{r,z})}W(\mathrm dr\mathrm dz)\right|^{2p}\right]
	\\&\leq \Cr{:D.7}(p) \mathbf E_{\varepsilon \mathbf 1_U}\left[\left(\iint_0^\infty |G_{s'-r,y'-z}-G_{s-r,y-z}|^2 u_{r,z} \mathrm dr\mathrm dz\right)^{p}\right]
	\\&=  \Cr{:D.7}(p)\mathcal K_{s,y;s',y'}^p\mathbf E_{\varepsilon \mathbf 1_U}\left[\left(\frac{1}{\mathcal K_{s,y;s',y'}} \iint_0^\infty |G_{s'-r,y'-z}-G_{s-r,y-z}|^2 u_{r,z}\mathrm dr\mathrm dz\right)^{p}\right] 
	\\&\leq \Cr{:D.7}(p)\mathcal K_{s,y;s',y'}^p \mathbf E_{\varepsilon \mathbf 1_U}\left[\frac{1}{\mathcal K_{s,y;s',y'}} \iint_0^\infty |G_{s'-r,y'-z}-G_{s-r,y-z}|^2 u_{r,z}^p \mathrm dr\mathrm dz \right] 
	\\&\leq \Cr{:D.7}(p) \Cr{:D_6}^p  \varepsilon \left(|y'-y| + \sqrt{|s'-s|}\right)^p.
\end{align}
	Here, in the last inequality, we used the fact that $\mathbf E_{\varepsilon \mathbf 1_U}[u_{r,z}^p]\leq \mathbf E_{\varepsilon \mathbf 1_U}[u_{r,z}] \leq \varepsilon$ for every $r\geq 0$ and $z\in \mathbb R$.
\begin{note}
	Also notice that 
\[
	|x-y|+\sqrt{|t-s|} \leq (\sqrt{2n}+1)(\sqrt{|x-y|} + \sqrt{|t-s|})\leq 2(\sqrt{2n}+1) \sqrt{|x-y| + |t-s|}
\]
	provided $n>0$, $t,s\geq 0$ and $x,y\in [-n,n]$. 
	To summarize, for every $p\geq 1$ and $n>0$, there exists a constant $\C\label{:D.8}(p,n)>0$, depending only on $p$ and $n$, such that for every $t,s\geq 0$ and $x,y\in [-n,n]$, 
\begin{equation} 
	\mathbb E[|u_{t,x}-u_{s,y}|^{2p}] 
	\leq \Cr{:D.8}(p,n)  \varepsilon (|x-y| + |t-s|)^{p/2}.
\end{equation}
\end{note}
	Now, from Kolmogorov continuity theorem for random fields \cite{MR876085}*{Corollary 1.2}, there exists a constant $\C\label{:D.9}(t,p,F)>0$, depending only on $t$, $p$ and the bounded $F$, such that 
\[
	\mathbf E_{\varepsilon \mathbf 1_U} \left[\sup_{(s',y'),(s,y)\in [0,t]\times F}|N_{s'}(y')-N_{s}(y)|^{2p}\right] 
	\leq \Cr{:D.9}(t,p,F)\varepsilon.
\]
	Finally, for some $y_0\in F$, by Markov's inequality and the fact that $N_0(y_0) = 0$,
\begin{align} 
	&\mathbf P_{\varepsilon \mathbf 1_U}\left(\sup_{(s,y)\in [0,t]\times F} |N_{s}(y)|\geq \frac{\gamma}{2}\right)
	\leq \mathbf P_{\varepsilon \mathbf 1_U}\left(\sup_{(s,y)\in [0,t]\times F} |N_{s}(y)-N_{0}(y_0)|\geq \frac{\gamma}{2}\right)
	\\&\leq \frac{\mathbf E_{\varepsilon \mathbf 1_U} \left[\sup_{(s,y)\in [0,t]\times F} |N_{s}(y)-N_0(y_0)|^{2p}\right]}{(\gamma/2)^{2p}} 
	\leq \frac{\Cr{:D.9}(t,p,F)}{(\gamma/2)^{2p}}\varepsilon.
\end{align}
	The desired result \eqref{:IFP_30} in Case (4) now follows.
\end{proof}
\begin{note}
	\begin{issue}
		We might not need this note depending on Clayton's new proof.
	\end{issue}
\begin{proof}[Proof of Lemma \ref{:BS}]
	For fixed arbitrary $T\geq 0$ and $a\in \mathbb R$,  we notice that
\begin{align}
	&L_{T,a}
	:=\sum_{i\in \mathbb N} \mathbf 1_{\{\exists t\in [0,T]: X_i+\beta^{(i)}_t \in [a,a+1)\}}
	\\&= \sum_{i\in \mathbb N}  \left(\sum_{k\in \mathbb Z}\mathbf 1_{\{a-X_i\in [k-1,k)\}}\right) \mathbf 1_{\{\exists t\in [0,T]: \beta^{(i)}_t \in [a-X_i,a-X_i+1)\}}
	\\&\leq\sum_{k\in \mathbb Z} \sum_{i\in \mathbb N}  \mathbf 1_{\{X_i \in [a-k,a-k+1)\}} \mathbf 1_{\{\exists t\in [0,T]: \beta^{(i)}_t \in [k-1,k+1]\}}.
\end{align}
	Taking expectation, and using the fact that the independent Brownian motions $\{(\beta^{(i)}_t)_{t\geq 0}:i\in \mathbb N\}$ are independent of the random sequence $(X_i)_{i\in \mathbb N}$, we have
\begin{align}
	&\mathbb E[L_{T,a}]
	\leq \sum_{k\in \mathbb Z} \mathbb E\left[\sum_{i\in \mathbb N}\mathbf 1_{\{X_i \in [a-k,a-k+1)\}}\right] \mathbb P\left(\exists t\in [0,T]: \beta^{(1)}_t \in [k-1,k+1]\right)
	\\& \leq \left(\sup_{l\in \mathbb R}\mathbb E\left[\sum_{i\in \mathbb N}\mathbf 1_{\{X_i \in [l,l+1)\}}\right]\right) 2\sum_{k=0}^\infty    \mathbb P\left(\exists t\in [0,T]: \beta^{(1)}_t \in [k-1,k+1] \right)
	\\& \leq \left(\sup_{l\in \mathbb R}\mathbb E\left[\sum_{i\in \mathbb N}\mathbf 1_{\{X_i \in [l,l+1)\}}\right]\right) 2 \sum_{k=0}^\infty   \mathbb P\left(\sup_{t\in [0,T]} \beta^{(1)}_t \geq k-1 \right) < \infty
\end{align}
	which implies the desired result.
\end{proof}
\end{note}

\subsection{Proof of Theorem \ref{thm:Markov}}
As we have mentioned in~Remark~\ref{rem:2311_01}, it is enough to proof Theorem \ref{thm:Markov}(i), since 
Theorem \ref{thm:Markov}(ii) is an immediate corollary of it.
\begin{proof}[Proof of Theorem \ref{thm:Markov} \emph{(i)}]
	From Theorem \ref{:I} \emph{\ref{:IF}}, we already know that $(Z_t)_{t>0}$ is a $\mathcal N$-valued process. 
		Let us now verify that it is a c\`adl\`ag process. 
		Fix an arbitrary $0< \delta < T < \infty$ and $l > 0$.
		From Proposition \ref{prop:finitecrossing}, we have $\mathbb E[Z((\delta, T)\times (-l,l))] < \infty$.
		This implies that almost surely there are only finitely many particle whose trajectory intersects with the time-space region $(\delta, T)\times (-l,l)$.
		Denote by $I_{(\delta, T)\times (-l,l)}$ the labels of those particles, i.e.
	\[
		I_{(\delta, T)\times (-l,l)} := \{i \in I_0: \exists t \in (\delta, T) \text{~s.t.~}  X^{(i)}_t \in (-l,l)\}.
	\]
		Then in other word, the event $\Omega_{\delta, T, l}:=\{ I_{(\delta, T)\times (-l,l)}\text{~has finite cardinality}\}$ has probability $1$.
		Notice that if $t\in (\delta, T)$ and $\varphi$ is a continuous testing function supported on $(-l,l)$, then
	\[
	\langle \varphi,Z_t \rangle = \sum_{i \in I_t} \varphi(X^{(i)}_t) = \sum_{i \in I_{(\delta, T)\times (-l,l)} } \varphi(X^{(i)}_t) 
	\]
	with a convention that $\varphi(\dagger) = 0$.
	Now on the event $\Omega_{\delta, T, l}$, for every continuous function $\varphi$ supported on $(-l,l)$, we have $(\langle \varphi,Z_t\rangle)_{\delta < t< T}$ is a c\`adl\`ag process because it is the sum of finitely many c\`adl\`ag processes.
	From this, it is straightforward to verify that $(Z_t)_{t>0}$ a $\mathcal N$-valued c\`adl\`ag process.
	
		We still needs to verify that \eqref{eq:EntranceLaw}  and \eqref{eq:TransitionProbability} hold for the process $(Z_t)_{t>0}$.
		Let us first consider the case when there are only finitely many initial particles, i.e. $\# I_0 < \infty$.
		Equivalently speaking in terms of the initial trace, $\Lambda = \varnothing$ and $\mu(\mathbb R) < \infty$.
		In this case \eqref{eq:EntranceLaw} holds for the process $(Z_t)_{t>0}$ because it degenerates to Shiga's duality formula \cite{MR948717}*{Theorem 5.2}
\[
		\mathbb E\left[ \prod_{i\in I_t} \left( 1-f(X^{(i)}_t) \right) \right] = \mathbf E_f\left[ \prod_{i\in I_0} (1-u_{t,x_i}) \right].
\]	
	Similarly, \eqref{eq:TransitionProbability} also holds for the process $(Z_t)_{t>0}$ and filtration $(\mathscr F_t)$ by applying Shiga's duality to conditional expectation.

	Let us now verify that \eqref{eq:EntranceLaw} holds for the process $(Z_t)_{t>0}$ when $\# I_0 = \infty$.
	Notice that in this case for every $n\in \mathbb N$,  $\mathbf X^{(n)} := \{X^{(i)}_t: t\geq 0, i = 1,\dots, n\}$ is a coalescing Brownian particle system with initial configuration $(x_i)_{i=1}^n$. 
	Denote by $I_t^{(n)}$ the random set of the index of the particles alive at a given time $t > 0$ in the particle system $\mathbf X^{(n)}$.
	Let $f$ be an arbitrary $[0,1]$-valued Borel measurable function on $\mathbb R$.
	From Shiga's duality \cite{MR948717}*{Theorem 5.2}, 
	\[
	\ex \left[\prod_{i \in I_t^{(n)}}\left(1 - f(  X_t^{(i)}  )\right)\right] 
	= \mathbf E_{f} \left[\prod_{i = 1}^n\left(1 - u_{t, x_i}\right)\right], \quad n\in \mathbb N.
	\]
	Taking $n\uparrow\infty$, we get by monotone convergence theorem that
	\begin{equation} \label{eq:ifduality}
	\mathbb E\left[\prod_{i\in I_t}\left(1-f(X_t^{(i)})\right)\right] 
	= \mathbf E_{f}\left[\prod_{i=1}^\infty (1-u_{t,x_i})\right].
	\end{equation}
	From the fact that $u_{t,x}$ is continuous in $x$, we can verify the following analytical fact
\[
\prod_{i=1}^\infty (1-u_{t,x_i}) = \mathbf 1_{\{u_{t,x} = 0, \forall x\in \Lambda\}} \prod_{i\in \mathbb N: x_i \notin \Lambda}  (1-u_{t,x_i}).
\]
	From this, we can rewrite \eqref{eq:ifduality} into 
		\begin{equation}
		\mathbb E[ \exp\left\{\langle \log (1-f), Z_t \rangle\right\}] = \mathbf E_f[\mathbf 1_{\{u_{t,x} = 1,\forall x\in \Lambda\}} \exp\left\{\langle \log (1-u_t),\mu\rangle \right\}]
	\end{equation}
as desired.

	Similarly, we can verify that \eqref{eq:TransitionProbability} holds for process $(Z_t)_{t>0}$ and filtration $(\mathscr F_t)_{t\geq 0}$.
	In fact, by applying Shiga's duality
		with conditional expectation to the particle system $\mathbf X^{(n)}$, we have
	\[
	\mathbb E\left[\prod_{i\in I^{(n)}_t} \left( 1-f(X_t^{(i)}) \right) \middle| \mathscr F_s\right] 
	= \int \left( \prod_{i\in I^{(n)}_s}\left(1-\omega(X_s^{(i)}) \right) \right) \mathbf P_f(u_{t-s} \in  \mathrm d\omega), \quad n\in \mathbb N.
	\]
	Here, $ \mathbf P_f(u_{t-s} \in  \mathrm d\omega)$ is the law of the random function $u_{t-s}$ under the probability $\mathbf P_f$.
	Taking $n\to \infty$, we get 
	\[
	\mathbb E\left[\prod_{i\in I_t} \left( 1-f(X^{(i)}_t) \right) \middle| \mathscr F_s\right] 
	= \int \left( \prod_{i\in I_s}\left(1-\omega( X_s^{(i)}) \right) \right) \mathbf P_f(u_{t-s} \in  \mathrm d\omega)
	\]
which can be rewritten into 
\begin{equation}
	\mathbb E\left[\exp\{\langle \log(1-f), Z_t\rangle \}\middle| \mathscr F_s \right] 
	= \Theta^f_{t-s} (Z_s).
\end{equation}	
	This also implies the Markov property of $(Z_t)_{t>0}$.
\end{proof}

\section{Proof of Theorem \ref{:E}} \label{:E:}

	From Theorem \ref{thm:Markov}, we know that a coalescing Brownian motions process with initial trace $(\Lambda,\mu)$ can be realized by the process $(Z_t)_{t> 0}$ constructed through \eqref{:.01}--\eqref{:filtration} by taking a list of real numbers $(x_i)_{i\in I_0} \in \mathcal X$ so that $(\Lambda,\mu) = \Psi((x_i)_{i\in I_0})$. 
	Therefore, we only have to prove that the statements \emph{\ref{:StayInfinity}}--\emph{\ref{:EF}} hold for this precise realization $(Z_t)_{t> 0}$.
	First notice that this is trivial if $\# I_0 < \infty$. 
	In fact, if $\# I_0 < \infty$, then $\Lambda = \varnothing$ and $\mu = \sum_{i = 1}^n \delta_{x_i}$ for some finite $n$. 
	In this case, statements \emph{\ref{:StayFinite}}--\emph{\ref{:EI}} hold for the process $(Z_t)_{t>0}$ simply because $Z_t(U) \leq n$ for any $t\geq 0$ and any open interval $U$; and there is no open interval $U$ satisfying the conditions of statements \emph{\ref{:StayInfinity}} and \emph{\ref{:EF}}.
	Therefore, in the rest of this Section, we will assume that $I_0 = \mathbb N$.
	This also allows us to use the lemmas from Subsection \ref{:IF:}.

		Let us first list some lemmas whose proofs are postponed at the end of this section. 
		\begin{lemma} \label{:UF}
	Let $F$ be a closed interval.
	\begin{enumerate}
	\item 
	If $U$ is an open interval such that $U\cap F$ is bounded then 
		\[
			\limsup_{t\downarrow 0} \Cr{:IFP.31}(U,F,\gamma,t) < \infty, \quad \gamma \in (0,1)
		\]
			where $\Cr{:IFP.31}(U,F,\gamma,t)$ is given as in Lemma \ref{:IFP_3}. 
	\item
	If $F$ contains $\cup_{i\in \mathbb N}(x_i-1,x_i+1)$, then
	\[
		\limsup_{t\downarrow 0} \mathcal V_t^{(\Lambda, \mu, F)} < \infty.
	\]
	\end{enumerate}
	\end{lemma}

\begin{lemma} \label{:EO}
	Let $U$ be an open interval such that  $U \cap \supp(\Lambda,\mu)$ is bounded.
\begin{enumerate}
\item \label{:EO1}
	If $\bar U\cap \Lambda = \varnothing$, then $\limsup_{t\downarrow 0} \int_U v^{(\Lambda, \mu)}_{t,x}\mathrm dx < \infty$. 
\item \label{:EO2}
	If $\bar U\cap \Lambda \neq \varnothing$, then $\lim_{t\downarrow 0} \int_U v^{(\Lambda, \mu)}_{t,x}\mathrm dx = \infty$. 
\end{enumerate}
\end{lemma}

	\begin{lemma}\label{:EE}
	Let $U\subset \mathbb R$ be an open interval such that $U \cap \supp(\Lambda,\mu)$ is bounded.
	Suppose that $\bar U \cap \Lambda \neq \varnothing$. 
	Then as $t\downarrow 0$,
\[
	\left(\int_U v^{(\Lambda, \mu)}_{t, x} \mathrm dx\right)^{-1}\mathbb E[Z_t(U)] \longrightarrow 1.
\]
\end{lemma}
\begin{lemma}\label{:EP}
	Let $U\subset \mathbb R$ be an open interval such that $U \cap \supp(\Lambda,\mu)$ is bounded.
	Suppose that $\bar U \cap \Lambda \neq \varnothing$. 
	Then as $t\downarrow 0$,
\[
	\left(\int_U v^{(\Lambda, \mu)}_{t, x} \mathrm dx\right)^{-1}Z_t(U) \longrightarrow 1 \quad \text{in probability}.
\]
\end{lemma} 

\begin{proof}[Proof of Theorem \ref{:E}]
	As mentioned at the beginning of this section, we only have to show that the statements \emph{\ref{:StayInfinity}}--\emph{\ref{:EF}} hold for the process $(Z_t)_{t>0}$ with the assumption that $\#I_0 = \infty$. 
	For statements \emph{\ref{:StayInfinity}} and \emph{\ref{:StayFinite}}, this is done in Theorem \ref{:I}.
	For statements \emph{\ref{:EK}}, this is done in Proposition \ref{:IF_fixed_t} and Remark \ref{thm:Finitem}.
	
	Let us now verify that statement \emph{\ref{:EI}} holds for the process $(Z_t)_{t>0}$.
	Let $\gamma \in (0,1)$ be arbitrary.
	Let $F$ be the smallest closed interval containing $\cup_{i\in \mathbb N} (x_i-1,x_i+1)$.
	From the condition that $U \cap \supp(\Lambda,\mu)$ is bounded, we immediately get that $U \cap F$ is bounded. 
	From Lemma \ref{:UF} we know that both $\limsup_{t\downarrow 0} \mathcal V_t^{(\Lambda, \mu, F)}$ and $\limsup_{t\downarrow 0}\Cr{:IFP.31}(U, F, t, \gamma)$ are finite. 
	With this at hand, the desired result follows from Proposition \ref{:IF_fixed_t} and Lemma \ref{:EO} \eqref{:EO1}.  
	
	Finally, from Lemmas \ref{:EO} \eqref{:EO2}, \ref{:EE}, \ref{:EP} and \cite{MR4226142}*{Theorem 5.12}, we can verify that the statement \emph{\ref{:EF}} holds for the process $(Z_t)_{t>0}$.
\end{proof}

	Let us now give the proofs of Lemmas \ref{:UF}--\ref{:EP}.
	\begin{proof}[Proof of Lemma \ref{:UF}]
		From the definition of $(\mathcal V_t^{(\Lambda, \mu, F)})_{t>0}$ and $(\Cr{:IFP.31}(U, F, t, \gamma))_{t>0}$ (see \eqref{:IFP_15} and Lemma \ref{:IFP_3} respectively),
	we know that they are non-decreasing in $t>0$.
		Now the desired results follow from Lemma \ref{:IFP_3} and Lemma \ref{:IFP_2} (2). 
	\end{proof}
\begin{proof}[Proof of Lemma \ref{:EO} \eqref{:EO1}]
	Since $U$ is an open interval, there are four cases to consider.

	\emph{Case (1), $U = \mathbb R$}: 
	This case won't happen, because from the condition of the lemma we have $\supp(\Lambda, \mu)$ is bounded and $\Lambda = \varnothing$. 
	Now $\mu$ is a locally finite measure on $\mathbb R$ with compact support, and in particular $\#I_0 = \mu(\mathbb R) < \infty$.
	This contradicts the assumption we made at the beginning of this section.
	
	\emph{Case (2), $U = (\alpha,\beta)$ for $-\infty < \alpha \leq \beta < \infty$}:
	From the condition that $\bar U \cap \Lambda = \varnothing$, we  get that the closed interval $\bar U = [\alpha, \beta]$ is contained in the open set $\Lambda^c$.
	Therefore, there exists a small $\delta>0$ such that $[\alpha-\delta, \beta+\delta]$ is also contained in $\Lambda^c$.
	It can be verified that 
 \[Z_0([\alpha-\delta,\beta+\delta]) = \#(\{x_i:i\in \mathbb N\}\cap [\alpha-\delta,\beta+\delta])<\infty, \] 
 since if it is infinite then there exists a limiting accumulation point of  $\{x_i:i\in \mathbb N\}$ in $\Lambda^c$, which contradicts \eqref{:lambda}.
	Let us take a continuous $[0,1]$-valued function $\phi$ supported on $[\alpha-\delta,\beta+\delta]$ such that $\phi = 1$ on $[\alpha,\beta]$.
	Now as $t\downarrow 0$, from \eqref{:.05} we have
\[
	\int_\alpha^\beta v_{t,x}^{(\Lambda,\mu)}\mathrm dx 
	\leq \int_{\Lambda^c} v^{(\Lambda,\mu)}_{t,x} \phi(x)\mathrm dx 
	\longrightarrow \sum_{i\in \mathbb N} \phi(x_i)
	\leq Z_0([\alpha-\delta,\beta+\delta]) < \infty
\]
	as desired for this case.
	
	\emph{Case (3), $U = (\alpha,\infty)$ for some $\alpha \in \mathbb R$}:
	From the condition that $U \cap \supp(\Lambda,\mu)$ is bounded, one can get that there exists $\beta \in U$ big enough, such that $(\beta-1,\infty) \cap \supp(\Lambda,\mu) = \varnothing$.
	It is then clear that $(\Lambda, \mu)\preceq ((-\infty, \beta-1],\mathbf 0)$.
	Therefore from \eqref{:II.522}, \eqref{:II.524} and \eqref{:II.526}, we have
\begin{align}
	&\int_{\beta}^\infty v_{t,x}^{(\Lambda, \mu)}\mathrm dx
	\leq \int_{\beta}^\infty v_{t,x}^{((-\infty,\beta-1], \mathbf 0)}\mathrm dx
	= \int_{1}^\infty v_{t,x}^{((-\infty,0], \mathbf 0)}\mathrm dx
	\\&\leq \frac{\Cr{:II.53}}{\sqrt{t}}\int_{t^{-1/2}}^\infty \left(1+z\right)e^{-\frac{1}{2}z^2} \mathrm dz
	\leq \frac{\Cr{:II.53}}{\sqrt{t}}\int_{t^{-1/2}}^\infty \left(\frac{z}{t^{-1/2}}+z\right)e^{-\frac{1}{2}z^2} \mathrm dz
	\\\label{:EPOO1}&= \Cr{:II.53}\left(1+\frac{1}{\sqrt{t}}\right)e^{-\frac{1}{2t}}
	\longrightarrow 0,
	\qquad t\downarrow 0.
\end{align}
	Also note that, by taking $\tilde U := (\alpha,\beta)$, we have $\tilde U \cap \supp(\Lambda,\mu)$ is bounded and $\tilde U \cap \Lambda = \varnothing$.
	So from what we proved in Case 2, we know that $\limsup_{t\downarrow 0} \int_\alpha^\beta v_{t,x}^{(\Lambda,\mu)}\mathrm dx < \infty$.
	Combining this with \eqref{:EPOO1}, we get the desired result for this case.

	\emph{Case (4), $U = (-\infty,\beta)$ for some $\beta \in \mathbb R$}: 
	This is the same as Case 3, thanks to the spatial symmetry.
\end{proof}
\begin{proof}[Proof of Lemma \ref{:EO} \eqref{:EO2}]
	Since $\overline U \cap \Lambda \neq \varnothing$, there exists $\tilde x\in \overline U\cap \Lambda$.
	By shifting $(\Lambda, \mu)$ and $U$ together, we can assume without loss of generality that $\tilde x = 0$.
	Now since $U$ is an open interval whose closure contains $0$, we know that there exists $\delta > 0$, such that either $(-\delta,0)\subset U$ or $(0,\delta)\subset U$ holds.
	Define set $-U:=\{-y\in \mathbb R: y \in U \}$, it is then clear that $(-\delta,0)\cup (0,\delta) \subset U \cup (-U)$.  
	Now from \eqref{:.05}, \eqref{:II.522}, \eqref{:II.525} and the fact that $(\{0\},\mathbf 0) \preceq (\Lambda, \mu)$, we have as $t\downarrow 0$,
\begin{align}
	&\int_U v^{(\Lambda,\mu)}_{t,x} \mathrm dx \geq \int_U v^{(\{0\},\mathbf 0)}_{t,x} \mathrm dx
	= \frac{1}{2} \Big(\int_U v^{(\{0\},\mathbf 0)}_{t,x} \mathrm dx + \int_{-U} v^{(\{0\},\mathbf 0)}_{t,x} \mathrm dx\Big)
	\\&\geq \frac{1}{2} \int_{U\cup (-U)} v^{(\{0\},\mathbf 0)}_{t,x} \mathrm dx
	\geq \frac{1}{2} \int_{-\delta}^\delta v^{(\{0\},\mathbf 0)}_{t,x} \mathrm dx
	\longrightarrow +\infty
\end{align}
	as desired.
\end{proof}

\begin{proof}[Proof of Lemma \ref{:EE}]
	Let $F$ be the smallest closed interval containing $\cup_{i\in \mathbb N} (x_i-1,x_i+1)$.
	Let $0<\gamma < 1$ be arbitrary.
	From the condition that  $U \cap \supp(\Lambda, \mu)$ is bounded, we know that $U\cap F$ is also bounded.
	From Proposition \ref{:IF_fixed_t} we know \eqref{:IF_fixed_t1} holds for any time $t>0$. 
	From Lemma \ref{:UF}, we know that both $\limsup_{t\downarrow 0} \mathcal V_t^{(\Lambda, \mu, F)}$ and $\limsup_{t\downarrow 0}\Cr{:IFP.31}(U, F, t, \gamma)$ are finite. 
	These and Lemma \ref{:EO} \eqref{:EO2} easily imply that
\begin{align}
	&\limsup_{t\downarrow 0}\Big(\int_Uv^{(\Lambda, \mu)}_{t, y} \mathrm d y\Big)^{-1}\ex[Z_t(U)] 
	\leq \frac{1}{1-\gamma}.
\end{align}
	Since $\gamma \in (0,1)$ is arbitrary, we can replace the right hand side of the above inequality by $1$. 
	On the other hand, by Lemmas \ref{:IFP_1} (see \eqref{:IFP_12}) and \ref{:IFP_25} (see \eqref{:IFP_252}) 
			for any $t>0$ we have 
\begin{align}
	&\ex[Z_t(U)] 
	= \lim_{\ve \downarrow 0}\frac{1 - \ex\left((1- \ve)^{Z_t(U)}\right)}{\ve}
	\\&\geq \lim_{\ve \downarrow 0} \ve^{-1} \left( 1 - \exp\left(-\varepsilon \int_U v^{(\Lambda,\mu)}_{t,y} \mathrm dy\right)\right)
	=  \int_U v^{(\Lambda,\mu)}_{t,y} \mathrm dy.
\end{align}
	This implies that
\[
	\liminf_{t\downarrow 0} \Big(\int_Uv^{(\Lambda, \mu)}_{t, y} \mathrm d y\Big)^{-1}\ex[Z_t(U)] 
	\geq 1.
\]
	Thus the desired result follows.
\end{proof}

\begin{proof}[Proof of Lemma \ref{:EP}]
	Let $\vartheta > 0$ be arbitrary.
	Define	
\[
	\varepsilon(U,\vartheta,t)
	:= 1 - \exp\left(- \left(\int_U v^{(\Lambda,\mu)}_{t,x}\mathrm dx\right)^{-1}\vartheta\right), 
	\quad t>0.
\]
	From Lemma \ref{:EO} \eqref{:EO2}, it is easy to see that
\[
	\varepsilon(U,\vartheta,t)\int_U v_{t,y}^{(\Lambda,\mu)}\mathrm dy \longrightarrow \vartheta, \quad \text{as} \quad t\downarrow 0.
\]
	Also notice that
\begin{equation}\label{:Finally}
	\mathbb E\left[\exp\left(-\vartheta \left(\int_U v^{(\Lambda,\mu)}_{t,x}\mathrm dx\right)^{-1} Z_t(U)\right)\right]
	= \mathbb E\left[\left(1-\varepsilon(U,\vartheta,t)\right)^{Z_t(U)}\right],
	\quad t>0.
\end{equation}
By Lemmas \ref{:IFP_1} (see \eqref{:IFP_12}) and \ref{:IFP_25} (see \eqref{:IFP_252}), we have 
\begin{align}
	\mathbb E\left[\left(1-\varepsilon(U,\vartheta,t)\right)^{Z_t(U)}\right]
	\leq \exp\left(-\varepsilon(U,\vartheta,t)\int_U v^{(\Lambda,\mu)}_{t,y}\mathrm dy\right),
	\quad t>0,
\end{align}
	which implies that
\begin{equation} \label{:Last}
	\limsup_{t\downarrow 0}\mathbb E\left[\left(1-\varepsilon(U,\vartheta,t)\right)^{Z_t(U)}\right]
	\leq e^{-\vartheta}.
\end{equation}
	Let us now take $F$ to be the smallest closed interval containing  $\cup_{i\in \mathbb N}(x_i-1,x_i+1)$ and take an arbitrary $\gamma \in (0,1)$. 
		From Lemma \ref{:EO} \eqref{:EO2}, we have
\begin{equation}\label{:LV}
	\varepsilon (U,\vartheta,t) \to 0, \quad \text{~as~} t\downarrow 0.
\end{equation}
	and therefore, there exists $t_0(U,\vartheta,\gamma)$ such that $2\varepsilon(U,\vartheta,t)< \gamma$ for every $0<t\leq t_0(U,\vartheta,\gamma)$.
	Now by Lemmas \ref{:IFP_1} (see \eqref{:IFP_11}) and \ref{:IFP_25} (see \eqref{:IFP_251}), for every $0<t\leq t_0(U,\vartheta,\gamma)$, we have

\begin{align}
	&\mathbb E\left[\left(1-\varepsilon(U,\vartheta,t)\right)^{Z_t(U)}\right]
	\\&\geq \exp\left(-\frac{\varepsilon(U,\vartheta,t)}{1-\gamma}\int_U v^{(\Lambda,\mu)}_{t,y}\mathrm dy\right)  - \frac{\varepsilon(U,\vartheta,t)}{2(1-\gamma)} \mathcal V_t^{(\Lambda,\mu,F)} - 2\mathbf P_{\varepsilon(U,\vartheta,t)\mathbf 1_U}\left(\sup_{s\leq t,y\in F} u_{s,y}>\gamma\right).
\end{align}
	Noticing that, by \eqref{:LV} and Lemma \ref{:UF}, 
\[
	\frac{\varepsilon(U,\vartheta,t)}{2(1-\gamma)} \mathcal V_t^{(\Lambda,\mu,F)} \longrightarrow 0,
	\quad \text{~as~} t\downarrow 0,
\]
	and that by Lemma \ref{:IFP_3} 
\[
	\mathbf P_{\varepsilon(U,\vartheta,t)\mathbf 1_U}\left(\sup_{s\leq t,y\in F} u_{s,y}>\gamma\right)
	\leq \varepsilon(U,\vartheta,t)\Cr{:IFP.31}(U, F, t, \gamma)
	\longrightarrow 0,
	\quad \text{~as~} t\downarrow 0,
\]
we can verify 
\[
	\liminf_{t\downarrow 0}\mathbb E\left[\left(1-\varepsilon(U,\vartheta,t)\right)^{Z_t(U)}\right]
	\geq e^{-\frac{\vartheta}{1-\gamma}}.
\]
	Since $\gamma\in (0,1)$ is arbitrary, we can replace the right hand side of the above inequliaty by $e^{-\vartheta}$.  
	From this, \eqref{:Finally} and \eqref{:Last}, we have
\[
	\lim_{t\downarrow 0}\mathbb E\left[\exp\left(-\vartheta \left(\int_U v^{(\Lambda,\mu)}_{t,x}\mathrm dx\right)^{-1} Z_t(U)\right)\right]
	= e^{-\vartheta}.
\]
Since $\vartheta>0$ is arbitrary, we are done.
\end{proof}

\section{Behavior of the rate function} \label{:B:}
	
	For the proof of Proposition \ref{:B}, we will use Le Gall's probabilistic representation \cite{MR1429263}*{Theorem 4} for the solutions of the PDE \eqref{:.05}.
	In the lemma below, we give a weak version of \cite{MR1429263}*{Theorem 4} avoiding the technical details of the Brownian snake. 
	We include its proof later in this section for the sake of completeness. 
	Denote by $\mathbb K$ the collection of compact subsets of $\mathbb R$ equipped with the Hausdorff metric.
	Denote by $\mathcal K$ the identity map on $\mathbb K$.
	Notice that $\mathcal K$ is a $\mathbb K$-valued random element on the measurable space $(\mathbb K,\mathscr B(\mathbb K))$.
	Define $A - \tilde A := \{a-\tilde a: a\in A, \tilde a\in \tilde A\}$ for any subsets $A$ and $\tilde A$ of $\mathbb R$.
\begin{lemma}\label{:BL}
	There exists a unique probability measure $\mathsf P$ on $\mathbb K$ such that 
\begin{equation} 
	\mathsf P(\mathcal K\cap A \neq \varnothing) = v_{2,0}^{(A,\mathbf 0)}, \quad A \in \mathbb K.
\end{equation}
	Furthermore, the following statements hold. 
\begin{enumerate}
\item
	$v^{(A,\mathbf 0)}_{t,x} 
	= 2t^{-1}\mathsf P\left(x\in \left(A - \sqrt{t/2}  \mathcal K\right)\right)$ for every $t>0$, $x\in \mathbb R$ and compact $A \subset \mathbb R$.
\item
	Extending the probability space $(\mathbb K, \mathscr B(\mathbb K), \mathsf P)$ if necessary, there exist real valued random variables $Y,\tilde Y$, and strictly positive random variables $R,\tilde R$, such that $\mathsf E [\tilde R] < \infty$ and that $\mathsf P\big( (Y-R,Y+R) \subset -\sqrt{1/2}\mathcal K \subset (\tilde Y-\tilde R,\tilde Y+\tilde R) \big) = 1.$
\end{enumerate}
\end{lemma}
	Here $\mathsf E$ is the corresponding expectation of the probability $\mathsf P$.
\begin{proof}[Proof of Proposition \ref{:B}]
	According to Lemma \ref{:BL} (1) and Fubini's theorem, we have
\begin{equation} \label{:B.05}
	\|v^{(A,\mathbf 0)}_{t,\cdot}\|_{L^1(\mathbb R)} 
	= 2t^{-1} \mathsf E\left[\lambda\left(A - \sqrt{t/2} \mathcal K\right)\right], \quad t>0. 
\end{equation}
	Let random variable $Y, \tilde Y, R$ and $\tilde R$ be given as in Lemma \ref{:BL} (2).

	(1) Suppose that $A$ has positive finite cardinality. 
	Notice that in this case $\lambda(A - \sqrt{t/2}\mathcal K)/\sqrt{t}$ converges to $\#A\cdot\lambda(\sqrt{1/2}\mathcal K)$ almost surely as $t\downarrow 0$. 
	Also notice that the family of random variables $\{\lambda (A - \sqrt{t/2}\mathcal K)/\sqrt{t}:t>0\}$ is dominated by the integrable random variable $2\# A\cdot \tilde R$. 
	Therefore, from the dominated convergence theorem
\[
	\sqrt{t}\|v^{(A,\mathbf 0)}_{t,\cdot}\|_{L^1(\mathbb R)} 
	= 2 \mathsf E[\lambda(A - \sqrt{t/2} \mathcal K)/\sqrt{t}] \xrightarrow[]{} 2\mathsf E[\lambda(\sqrt{1/2}\mathcal K)]\cdot \# A, \quad t\downarrow 0. 
\]
	From \eqref{:B.05} we have that $\Cr{:.1}=\|v_{1,\cdot}^{(\{0\}, \mathbf 0)}\|_{L^1(\mathbb R)}= 2 \mathsf E[\lambda(\sqrt{1/2}\mathcal K)]$. 
	The desired result in this case now follows.

	(2) Suppose that $A$ has positive finite Lebesgue measure.
	Notice that in this case
\[
	t\|v^{(A,0)}_{t,\cdot}\|_{L^1(\mathbb R)} 
	= 2 \mathsf E[\lambda(A - \sqrt{t/2} \mathcal K)] \geq 2\lambda(A),
	\quad t> 0.
\]
	On the other hand, since $A$ is compact, we can verify from the monotone convergence theorem that 
\[
	t\|v^{(A,\mathbf 0)}_{t,\cdot}\|_{L^1(\mathbb R)}
	= 2 \mathsf E[\lambda(A - \sqrt{t/2} \mathcal K)] \leq 2 \mathsf E[\lambda(A + \sqrt{t} \tilde R B^o)] \to 2\lambda(A), 
	\quad t\downarrow 0, 
\]
	where $B^o:=(-1,1)$ is the centered open unit ball in $\mathbb R$.
	So we have $t\|v^{(A,0)}_{t,\cdot}\|_{L^1(\mathbb R)}\to 2\lambda(A)$ as $t\downarrow 0$.

	(3) Suppose that $A$ has Minkowski dimension $\delta \in (0,1)$. 
	Define $\C\label{:B.15}(A):=\inf\{n\geq 0: A \subset [-n,n]\}$.
	Let $\overline \delta \in (\delta,1)$ be arbitrary.
	Noticing that 
\begin{equation} \label{:B.16}
	\frac{\lambda(A + rB^o)}{r^{1-\overline \delta}} 
	= \exp\bigg((\log r) \Big(\frac{\log\lambda(A + rB^o)}{\log r}-(1-\overline \delta)\Big)\bigg)
	\xrightarrow[]{} 0, \quad r\downarrow 0, 
\end{equation}
	so there exists $\C\label{:B.2}(A, \overline \delta)>0$ such that $\lambda(A + rB^o)/r^{1-\overline \delta} \leq 1$ for every $r\in (0,\Cr{:B.2}(A,\overline \delta))$.
	Also notice that,  
\[
	\frac{\lambda(A + rB^o)}{r^{1-\overline \delta}} 
	\leq \frac{2(\Cr{:B.15}(A)+r)}{r^{1-\overline\delta}}
	\leq \frac{2\Cr{:B.15}(A)}{\Cr{:B.2}(A, \overline \delta)^{1-\overline\delta}} + 2r^{\overline \delta}, \quad r\geq \Cr{:B.2}(A,\overline \delta).
\]
	Therefore, there exists $\C\label{:B.3}(A,\overline \delta)>0$ such that $ \lambda(A + rB^o)/r^{1-\overline \delta} \leq \Cr{:B.3}(A,\overline \delta) + 2r^{\overline \delta}$ for every $r>0$.
	Now, we can verify that the family of random variables $\{\lambda(A + \sqrt{t}\tilde R B^o)/t^{(1-\overline\delta)/2}:t\in (0,1]\}$ is dominated by the integrable random variable $\Cr{:B.3}(A,\overline \delta)\tilde R^{1-\overline \delta}+2\tilde R$. 
	From \eqref{:B.05}, Proposition \ref{:BL} (2), \eqref{:B.16} and the dominated convergence theorem, we have
\[
	t^{\frac{1+\overline\delta}{2}}\|v_{t,\cdot}^{(A, \mathbf 0)}\|_{L^1(\mathbb R)}
	= 2\mathsf E\Big[\frac{\lambda(A-\sqrt{t/2}\mathcal K)}{(\sqrt{t})^{1-\overline \delta}}\Big] 
	\leq 2\mathsf E\Big[\frac{\lambda(A+\sqrt{t}\tilde RB^o)}{(\sqrt{t}\tilde R)^{1-\overline\delta}}\tilde R^{1-\overline\delta}\Big]\to 0, \quad t\downarrow 0. 
\]
	From this, and the fact that $\overline \delta \in (\delta,1)$ is arbitrary, we can verify that
\[
	\limsup_{t\downarrow 0} \frac{\log \|v_{t,\cdot}^{(A, \mathbf 0)}\|_{L^1(\mathbb R)}}{\log (1/t)} \leq \frac{1+\delta}{2}.
\]
	On the other hand, let $\underline \delta \in (0,\delta)$ be arbitrary.
	Notice that 
\begin{equation} 
	\frac{\lambda(A + rB^o)}{r^{1-\underline \delta}} 
	= \exp\bigg((\log r) \Big(\frac{\log\lambda(A + rB^o)}{\log r}-(1-\underline \delta)\Big)\bigg)
	\xrightarrow[]{} \infty, \quad r\downarrow 0. 
\end{equation}
	So from \eqref{:B.05}, Lemma \ref{:BL} (2) and Fatou's lemma, 
\begin{align}
	&\liminf_{t\downarrow 0}t^{\frac{1+\underline\delta}{2}}\|v_{t,\cdot}^{(A, \mathbf 0)}\|_{L^1(\mathbb R)}
	= \liminf_{t\downarrow 0}2\mathsf E\Big[\frac{\lambda(A-\sqrt{t/2}\mathcal K)}{(\sqrt{t})^{1-\underline \delta}}\Big] 
      	\\&\geq 2\mathsf E\Big[\liminf_{t\downarrow 0}\frac{\lambda(A+\sqrt{t} RB^o)}{(\sqrt{t} R)^{1-\underline\delta}} R^{1-\underline\delta}\Big]=\infty. 
\end{align}
	From this, and the fact that $\underline \delta \in (0,\delta)$ is arbitrary, we can verify that
\[
	\liminf_{t\downarrow 0} \frac{\log \|v_{t,\cdot}^{(A, \mathbf 0)}\|_{L^1(\mathbb R)}}{\log (1/t)} 
	\geq \frac{1+\delta}{2}.
\]
	The desired result in this case now follows.
\end{proof}
\begin{proof}[Proof of Lemma \ref{:BL}]
	Denote by $\mathbb W$ the space of pairs $(\zeta,w)$ where $\zeta\in [0,\infty)$ and $w=(w_t)_{t\geq 0}$ is an $\mathbb R$-valued continuous path such that $w_t = w_{t\wedge \zeta}$ for every $t\geq 0$.
	The space $\mathbb W$ is equipped with the metric
\[
	d_{\mathbb W}\big((\zeta,w),(\zeta',w')\big)
	= |\zeta-\zeta'|+\sup_{t\geq 0}|w_{t}-w'_{t}|, \quad (\zeta,w),(\zeta,w')\in \mathbb W.
\]
	Denote by $\mathcal C([0,\infty), \mathbb W)$ the space of $\mathbb W$-valued continuous path $(\zeta_s,(w_{s,t})_{t\geq 0})_{s\geq 0}$ equipped with the topology of local uniform convergence. 
	According to \cite{MR1429263}*{Theorem 4}, there exists a $\sigma$-finite measure $\mathbb N_0$ on $\mathcal C([0,\infty),\mathbb W)$, known as the excursion measure of the Brownian snake initiated at position $0$, such that $ v^{(A, \mathbf 0)}_{r,0} = 2\mathbb N_0(\mathscr S_r\cap A\neq \varnothing) $ for every closed $A\subset \mathbb R$ and $r>0$.
	Here, for each $r>0$, $\mathscr S_r:(\zeta_s,(w_{s,t})_{t\geq 0})_{s\geq 0}\mapsto \{w_{s,r}: s\geq 0, \zeta_s \geq r \}$ is  a measurable map from $\mathcal C([0,\infty), \mathbb W)$ to $\mathbb K$.
\begin{note}
	Notice that under $\mathbb N_0$, the continuous process $(\zeta_s)_{s\geq 0}$ has the ``law" of Ito's excursion measure of Brownian motion. 
	Therefore there exists $\overline s>0$ such that $\zeta_s = 0$ for $s\geq \overline s$.
	The graph $\mathscr G$ of the Brownian snake is defined as the image of the compact set $\{(s,t):s\in [0,\overline s],t\in[0,\zeta_s]\}$ under the continuous map $(s,t)\mapsto (t,w_{s,t})$.
	Therefore, $\mathscr G$ is compact.
	Notice that for every $r> 0$, $\mathscr S_r$ is isometric to the compact set $\mathscr G\cap \{(r,x):x\in \mathbb R\}$; so $\mathscr S_r$ must be compact.
\end{note}
	From \eqref{:II.523}, we have $2\mathbb N_0(\mathscr S_2 \neq \varnothing) = v^{(\mathbb R,\mathbf 0)}_{2,0}= 1$.
	This allows us to define a probability measure $\tilde{\mathbb N}_0$ on $\mathcal C([0,\infty), \mathbb W)$ so that $\mathrm d \tilde{\mathbb N}_0 = 2\mathbf 1_{\{\mathscr S_2 \neq \varnothing\}} \mathrm d\mathbb N_0$.
	Now we have 
\begin{equation}\label{:B.43}
	v^{(A, \mathbf 0)}_{2,0} 
	= \tilde{\mathbb N}_0(\mathscr S_2\cap A \neq \varnothing)
\end{equation}
	for every closed $A \subset \mathbb R$.
	This gives us the existence of the desired probability $\mathsf P$, which is the law of $\mathscr S_2$ under $\tilde {\mathbb N}_0$.
	The uniqueness of $\mathsf P$ follows from \cite{MR2132405}*{Theorem 1.13}. 

	(1) One can verify directly from the uniqueness of the solution to the PDE \eqref{:.05} that $v^{(A,\mathbf 0)}_{t,x}= \alpha^2 v^{(\alpha A,\mathbf 0)}_{\alpha^2t,\alpha x}$ for every $\alpha >0$, $t>0$, $x\in \mathbb R$ and closed $A\subset \mathbb R$. 
\begin{note}
	In fact, define $\tilde v_{t,x} =\alpha^2 v_{\alpha^2t,\alpha x}=\alpha^2 v^{(\alpha A,\mathbf 0)}_{\alpha^2t,\alpha x}$.
	Then
\[
	\partial_x \tilde v_{t,x}
	= \alpha^3\partial_y v_{s,y}|_{s=\alpha^2t,y=\alpha x},
	\quad \partial_x^2 \tilde v_{t,x}
	= \alpha^4 \partial_y^2 v_{s,y}|_{s=\alpha^2t,\alpha x},
\]
	and
\begin{align}
	&\partial_t \tilde v_{t,x} 
	= \alpha^4 \partial_s v_{s,y}|_{s=\alpha^2 t,\alpha x}
	= \alpha^4 (\frac{1}{2}\partial_y^2 v_{s,y}-\frac{1}{2}v_{s,y}^2)|_{s=\alpha^2 t,\alpha x}
	= \frac{1}{2}\partial_x^2 \tilde v_{t,x}-\frac{1}{2}\tilde v_{t,x}^2.
\end{align}
	Also for any $y\in A $,
\[
	\int_{y-r}^{y+r}\tilde v_{t,x}\mathrm dx = 
	\alpha^2 \int_{y-r}^{y+r} v_{\alpha^2t,\alpha x}\mathrm dx = 
	\alpha \int_{\alpha y-\alpha r}^{\alpha y+\alpha r} v_{\alpha^2t,z}\mathrm dz \to \infty, 
	\quad t\downarrow 0. 
\]
\end{note}
	Now, we can verify for every $t>0$, $x\in \mathbb R$ and closed $A \subset \mathbb R$ that
\begin{align}
	&v^{(A, \mathbf 0)}_{t,x} = v^{(A-\{x\}, \mathbf 0)}_{t,0} =2t^{-1} v^{(\sqrt{2/t}(A-\{x\}), \mathbf 0)}_{2,0}
	\\&= 2t^{-1}\mathsf P\Big(\mathcal K\cap \big(\sqrt{2/t}(A-\{x\})\big) \neq \varnothing\Big)
	=2t^{-1}\mathsf P\big(x\in (A - \sqrt{t/2}\mathcal K)\big).
\end{align}
\begin{note}
	In fact, $\mathcal K\cap \big(\sqrt{2/t}(A-\{x\})\big) \neq \varnothing$ iff $(\sqrt{t/2}\mathcal K)\cap (A - \{x\})\neq \varnothing$ iff \{$\exists k \in \sqrt{t/2}\mathcal K, l \in A$ such that $k=l-x$\} iff $x\in A- \sqrt{t/2}\mathcal K$.
\end{note}

	(2) According to \cite{MR1429263}*{Theorem 4}, there exists a non-negative continuous random field $(\tilde u_{t,x})_{t>0,x\in \mathbb R}$, known as the density of the super-Brownian motion constructed from the Brownian snake,  such that 
\[
	v^{(\varnothing,\theta\nu)}_{2,0}
	= 2\mathbb N_0\Big[ 1- \exp\Big(-\frac{\theta}{2}\int \tilde u_{2,y}\nu(dy)\Big)\Big] 
\]
	for every $\theta > 0$ and non-negative finite Radon measure $\nu$ on $\mathbb R$.
	Furthermore, from how it is constructed, one can verify that $\mathbb N_0$-almost everywhere, $(\tilde u_{2,y})_{y\in \mathbb R}$ is supported on $\mathscr S_2$. 
	Fixing a non-negative finite Radon measure $\nu$ with $\supp(\nu)=\mathbb R$, taking $\theta \uparrow \infty$, we get from above and \eqref{:II.521} that
\begin{align}
	&1=v^{(\mathbb R,\mathbf 0)}_{2,0} = 2\mathbb N_0\Big( \int \tilde u_{2,y}\nu(dy) > 0\Big) = 2\mathbb N_0(\exists y \in \mathbb R,  \tilde u_{2,y}>0)
	\\& = \tilde {\mathbb N}_0(\exists y\in \mathbb R, \tilde u_{2,y}>0)
	= \tilde {\mathbb N}_0\Big(\bigcup_{n\in \mathbb N} \{\tilde u_{2,y}>0,\forall y\in (q_n,q_n')\}\Big)
\end{align}
	where $\{(q_n,q_n'):n\in \mathbb N\}$ is a sequential arrangement of $\{(q,q')\in\mathbb Q^2:q<q'\}$.
	This allows us to define a random variable $N:=\inf\{n\in \mathbb N: \tilde u_{2,y}>0,\forall y\in (q_n,q_n')\}$ which is finite almost surely under $\tilde{\mathbb N}_0$.
	We then define an $\mathbb R$-valued random variable $Y$ and a $(0,\infty)$-valued random variable $R$ so that $(Y-R,Y+R)=-\sqrt{1/2}(q_N,q_N')$.
	Now $\tilde {\mathbb N}_0$-almost surely, $(Y-R,Y+R) \subset -\sqrt{1/2}\supp(\tilde u_{2,\cdot})\subset -\sqrt{1/2}\mathscr S_{2}$.

	On the other hand, from \eqref{:B.43}, \eqref{:II.525}, \eqref{:II.524} and \eqref{:II.526}, we can verify that
\begin{align}
	&\tilde{\mathbb N}_0(0\vee \sup \mathscr S_2 \geq n)
	= \tilde {\mathbb N}_0(\mathscr S_2 \cap [n,\infty) \neq \varnothing) 
	\\&= v^{([n,\infty),\mathbf 0)}_{2,0}
	= v_{2,0}^{((-\infty,-n],\mathbf 0)} 
	= v_{2,n}^{((-\infty,0],\mathbf 0)} 
	\leq \frac{1}{2}\Cr{:II.53}\Big(1+\frac{n}{\sqrt{2}}\Big)e^{-\frac{n^2}{4}}, 
	\quad n>0.
\end{align} 
	So we have $\tilde {\mathbb N}_0[|0\vee\sup \mathscr S_2|]<\infty$.
	Similarly, $\tilde {\mathbb N}_0[|0\wedge \inf \mathscr S_2|]<-\infty$.
\begin{note}
	For a non-negative random variable $X$, we have
\begin{align}
	&E[X] \leq P(X\in [0,1))+2P(X\in[1,2)) + 3P(X\in [2,3))+\cdots 
	\\&= P(X\geq 0) + P(X\geq 1) + P(X\geq 2)+\cdots 
\end{align} 
\end{note}
	We can then define a real valued random variable $\tilde Y$ and a $(0,\infty)$-valued random variable $\tilde R$ so that the interval $(\tilde Y-\tilde R,\tilde Y+\tilde R)=-\sqrt{1/2}(0\wedge \inf \mathscr S_2 - 1, 0\vee \sup \mathscr S_2 + 1)$.
	Notice that $\tilde {\mathbb N}_0$-almost surely, $ -\sqrt{1/2}\mathscr S_{2}\subset (\tilde Y-\tilde R,\tilde Y+\tilde R)$.
	 We are done. 
\end{proof}


\begin{bibdiv}
	\begin{biblist}
		\bib{MR1673235}{article}{
			author={Aldous, D. J.},
			title={Deterministic and stochastic models for coalescence (aggregation
				and coagulation): a review of the mean-field theory for probabilists},
			journal={Bernoulli},
			volume={5},
			date={1999},
			number={1},
			pages={3--48},
			issn={1350-7265},
			review={\MR{1673235}},
		}
\bib{MR2892958}{article}{
	author={Angel, O.},
	author={Berestycki, N.},
	author={Limic, V.},
	title={Global divergence of spatial coalescents},
	journal={Probab. Theory Related Fields},
	volume={152},
	date={2012},
	number={3-4},
	pages={625--679},
	issn={0178-8051},
	review={\MR{2892958}},
}
	\bib{MR1813840}{article}{
		author={Athreya, S.},
		author={Tribe, R.},
		title={Uniqueness for a class of one-dimensional stochastic PDEs using
			moment duality},
		journal={Ann. Probab.},
		volume={28},
		date={2000},
		number={4},
		pages={1711--1734},
		issn={0091-1798},
		review={\MR{1813840}},
	}
\bib{barnes2021effect}{article}{
	author={Barnes, C.}, 
	author={Mytnik, L.}, 
	author={Sun, Z.},
	title={Effect of small noise on the speed of reaction-diffusion equations with non-Lipschitz drift},
	eprint={https://doi.org/10.48550/arXiv.2107.09377},
	date={2021},
}
\bib{MR2599198}{article}{
	author={Berestycki, J.},
	author={Berestycki, N.},
	author={Limic, V.},
	title={The $\Lambda$-coalescent speed of coming down from infinity},
	journal={Ann. Probab.},
	volume={38},
	date={2010},
	number={1},
	pages={207--233},
	issn={0091-1798},
	review={\MR{2599198}},
}
	\bib{MR2574323}{book}{
	author={Berestycki, N.},
	title={Recent progress in coalescent theory},
	series={Ensaios Matem\'{a}ticos [Mathematical Surveys]},
	volume={16},
	publisher={Sociedade Brasileira de Matem\'{a}tica, Rio de Janeiro},
	date={2009},
	pages={193},
	isbn={978-85-85818-40-1},
	review={\MR{2574323}},
}
\bib{MR4235476}{article}{
	author={Biswas, N.},
	author={Etheridge, A.},
	author={Klimek, A.},
	title={The spatial Lambda-Fleming-Viot process with fluctuating
		selection},
	journal={Electron. J. Probab.},
	volume={26},
	date={2021},
	pages={Paper No. 25, 51},
	review={\MR{4235476}},
}
\bib{blath2022stochastic}{article}{
	author={Blath, J.}, 
	author={Hammer, M.}, 
	author={Nie, F.},
	title={The stochastic Fisher-KPP equation with seed bank and on/off branching coalescing Brownian motion},
	journal={Stochastics and Partial Differential Equations: Analysis and Computations},
	date={2022},
	pages={1--46},
}
\bib{MR2014157}{article}{
	author={Doering, C. R.},
	author={Mueller, C.},
	author={Smereka, P.},
	title={Interacting particles, the stochastic
		Fisher-Kolmogorov-Petrovsky-Piscounov equation, and duality},
	note={Stochastic systems: from randomness to complexity (Erice, 2002)},
	journal={Phys. A},
	volume={325},
	date={2003},
	number={1-2},
	pages={243--259},
	issn={0378-4371},
	review={\MR{2014157}},
}
\bib{MR3582808}{article}{
	author={Durrett, R.},
	author={Fan, W.-T.},
	title={Genealogies in expanding populations},
	journal={Ann. Appl. Probab.},
	volume={26},
	date={2016},
	number={6},
	pages={3456--3490},
	issn={1050-5164},
	review={\MR{3582808}},
}
\bib{MR1415234}{article}{
	author={Evans, S. N.},
	author={Fleischmann, K.},
	title={Cluster formation in a stepping-stone model with continuous,
		hierarchically structured sites},
	journal={Ann. Probab.},
	volume={24},
	date={1996},
	number={4},
	pages={1926--1952},
	issn={0091-1798},
	review={\MR{1415234}},
}
\bib{MR4278798}{article}{
	author={Fan, W.-T. L.},
	title={Stochastic PDEs on graphs as scaling limits of discrete
		interacting systems},
	journal={Bernoulli},
	volume={27},
	date={2021},
	number={3},
	pages={1899--1941},
	issn={1350-7265},
	review={\MR{4278798}},
}
\bib{MR3968719}{article}{
	author={Forien, R.},
	title={The stepping stone model in a random environment and the effect of
		local heterogneities on isolation by distance patterns},
	journal={Electron. J. Probab.},
	volume={24},
	date={2019},
	pages={Paper No. 57, 35},
	review={\MR{3968719}},
}
\bib{MR3846839}{article}{
	author={Hammer, M.},
	author={Ortgiese, M.},
	author={V\"{o}llering, F.},
	title={A new look at duality for the symbiotic branching model},
	journal={Ann. Probab.},
	volume={46},
	date={2018},
	number={5},
	pages={2800--2862},
	issn={0091-1798},
	review={\MR{3846839}},
}
\bib{MR2162813}{article}{
	author={Hobson, T.},
	author={Tribe, R.},
	title={On the duality between coalescing Brownian particles and the heat
		equation driven by Fisher-Wright noise},
	journal={Electron. Comm. Probab.},
	volume={10},
	date={2005},
	pages={136--145},
	issn={1083-589X},
	review={\MR{2162813}},
}
	\bib{MR4029158}{article}{
		author={Hughes, T.},
		author={Perkins, E.},
		title={On the boundary of the zero set of super-Brownian motion and its
			local time},
		language={English, with English and French summaries},
		journal={Ann. Inst. Henri Poincar\'{e} Probab. Stat.},
		volume={55},
		date={2019},
		number={4},
		pages={2395--2422},
		issn={0246-0203},
		review={\MR{4029158}},
	}
\bib{MR863723}{article}{
	author={Iwata, K.},
	title={An infinite-dimensional stochastic differential equation with
		state space $C({\bf R})$},
	journal={Probab. Theory Related Fields},
	volume={74},
	date={1987},
	number={1},
	pages={141--159},
	issn={0178-8051},
	review={\MR{863723}},
}
\bib{MR4226142}{book}{
	author={Kallenberg, O.},
	title={Foundations of modern probability},
	series={Probability Theory and Stochastic Modelling},
	volume={99},
	note={Third edition [of  1464694]},
	publisher={Springer, Cham},
	date={[2021] \copyright 2021},
	pages={xii+946},
	isbn={978-3-030-61871-1},
	isbn={978-3-030-61870-4},
	review={\MR{4226142}},
}
\bib{MR958288}{article}{
	author={Konno, N.},
	author={Shiga, T.},
	title={Stochastic partial differential equations for some measure-valued
		diffusions},
	journal={Probab. Theory Related Fields},
	volume={79},
	date={1988},
	number={2},
	pages={201--225},
	issn={0178-8051},
	review={\MR{958288}},
}
\bib{MR2977849}{book}{
	author={Lapidus, M. L.},
	author={van Frankenhuijsen, M.},
	title={Fractal geometry, complex dimensions and zeta functions},
	series={Springer Monographs in Mathematics},
	edition={2},
	note={Geometry and spectra of fractal strings},
	publisher={Springer, New York},
	date={2013},
	pages={xxvi+567},
	isbn={978-1-4614-2175-7},
	isbn={978-1-4614-2176-4},
	review={\MR{2977849}},
}
\bib{MR1207305}{article}{
	author={Le Gall, J.-F.},
	title={A class of path-valued Markov processes and its applications to
		superprocesses},
	journal={Probab. Theory Related Fields},
	volume={95},
	date={1993},
	number={1},
	pages={25--46},
	issn={0178-8051},
	review={\MR{1207305}},
}
\bib{MR1429263}{article}{
	author={Le Gall, J.-F.},
	title={A probabilistic approach to the trace at the boundary for
		solutions of a semilinear parabolic partial differential equation},
	journal={J. Appl. Math. Stochastic Anal.},
	volume={9},
	date={1996},
	number={4},
	pages={399--414},
	issn={1048-9533},
	review={\MR{1429263}},
}
\bib{MR2223040}{article}{
	author={Limic, V.},
	author={Sturm, A.},
	title={The spatial $\Lambda$-coalescent},
	journal={Electron. J. Probab.},
	volume={11},
	date={2006},
	pages={no. 15, 363--393},
	issn={1083-6489},
	review={\MR{2223040}},
	doi={10.1214/EJP.v11-319},
}
\bib{MR1697494}{article}{
	author={Marcus, M.},
	author={V\'{e}ron, L.},
	title={Initial trace of positive solutions of some nonlinear parabolic
		equations},
	journal={Comm. Partial Differential Equations},
	volume={24},
	date={1999},
	number={7-8},
	pages={1445--1499},
	issn={0360-5302},
	review={\MR{1697494}},
}
\bib{MR2132405}{book}{
	author={Molchanov, I.},
	title={Theory of random sets},
	series={Probability and its Applications (New York)},
	publisher={Springer-Verlag London, Ltd., London},
	date={2005},
	pages={xvi+488},
	isbn={978-185223-892-3},
	isbn={1-85233-892-X},
	review={\MR{2132405}},
}
\bib{MR2793860}{article}{
	author={Mueller, C.},
	author={Mytnik, L.},
	author={Quastel, J.},
	title={Effect of noise on front propagation in reaction-diffusion
		equations of KPP type},
	journal={Invent. Math.},
	volume={184},
	date={2011},
	number={2},
	pages={405--453},
	issn={0020-9910},
	review={\MR{2793860}},
}
\bib{MR983088}{article}{
	author={Reimers, M.},
	title={One-dimensional stochastic partial differential equations and the
		branching measure diffusion},
	journal={Probab. Theory Related Fields},
	volume={81},
	date={1989},
	number={3},
	pages={319--340},
	issn={0178-8051},
	review={\MR{983088}},
}
\bib{MR1271224}{article}{
	author={Shiga, T.},
	title={Two contrasting properties of solutions for one-dimensional
		stochastic partial differential equations},
	journal={Canad. J. Math.},
	volume={46},
	date={1994},
	number={2},
	pages={415--437},
	issn={0008-414X},
	review={\MR{1271224}},
}
\bib{MR948717}{article}{
	author={Shiga, T.},
	title={Stepping stone models in population genetics and population
		dynamics},
	conference={
		title={Stochastic processes in physics and engineering},
		address={Bielefeld},
		date={1986},
	},
	book={
		series={Math. Appl.},
		volume={42},
		publisher={Reidel, Dordrecht},
	},
	date={1988},
	pages={345--355},
	review={\MR{948717}},
}
\bib{MR1339735}{article}{
	author={Tribe, R.},
	title={Large time behavior of interface solutions to the heat equation
		with Fisher-Wright white noise},
	journal={Probab. Theory Related Fields},
	volume={102},
	date={1995},
	number={3},
	pages={289--311},
	issn={0178-8051},
	review={\MR{1339735}},
}
\bib{MR876085}{article}{
	author={Walsh, J. B.},
	title={An introduction to stochastic partial differential equations},
	conference={
		title={\'{E}cole d'\'{e}t\'{e} de probabilit\'{e}s de Saint-Flour, XIV---1984},
	},
	book={
		series={Lecture Notes in Math.},
		volume={1180},
		publisher={Springer, Berlin},
	},
	date={1986},
	pages={265--439},
	review={\MR{876085}},
}
	\end{biblist}
\end{bibdiv}
\end{document}